\documentclass[11pt,reqno, a4paper]{amsart}
\usepackage[numbers,sort&compress]{natbib}
\usepackage{amsfonts,bbm}
\usepackage{amssymb,color}
\usepackage{amssymb}
\usepackage{fancyhdr}
\usepackage[titletoc]{appendix}
\usepackage{enumitem}
\usepackage{amsgen}
\usepackage{amscd}
\usepackage{mathrsfs}
\usepackage{cases}
\usepackage{subcaption}
\usepackage{mathtools}
\usepackage{amsfonts}
\usepackage{amsthm}
\usepackage{newlfont}
\usepackage{graphicx}
\usepackage{float}
\usepackage{amsmath}
\usepackage{url}
\usepackage{multirow}
\usepackage[margin=1in, a4paper]{geometry}

\newtheorem{thm}{Theorem}[section]

\newtheorem{lem}[thm]{Lemma}

\newcommand{\Del}[1]{}

\numberwithin{equation}{section}

\begin{document}

\author{Xintong Li$^{a,*}$}
\author{Yongsheng Li$^{a}$}

\thanks{$^{a}$
School of Mathematics,
South China University of Technology,
Guangzhou 510640, China}
\thanks{
* Corresponding author.
      E-mail addresses: lxtmath1@126.com (X. Li),
       yshli@scut.edu.cn (Y. Li)}

\title[limit problem for KdV-B and mKdV-B]{Uniform Well-Posedness and Inviscid Limit for the KdV-Burgers and mKdV-Burgers Equations on $\mathbb{T}$}
                                         
\begin{abstract}\noindent
This article investigates uniform well-posedness and inviscid limit behavior for the periodic Korteweg-de Vries-Burgers (KdV-B) and modified Korteweg-de Vries-Burgers (mKdV-B) equations:
\[
\partial_t u + \partial_x^3 u - \varepsilon \partial_x^2 u = \partial_x(u^\alpha), \quad u(0) = \phi,
\]
where $\alpha = 2, 3$, $\varepsilon \in (0, 1]$ is the diffusion coefficient, and $u : \mathbb{R}^+ \times \mathbb{T} \to \mathbb{R}$ is real-valued.

For the KdV-B equation ($\alpha=2$), we establish unconditional uniform global well-posedness in $H^s(\mathbb{T})$ for $s \geq 0$, uniformly for all $\varepsilon \in [0,1]$, without relying on auxiliary function spaces. Furthermore, we prove that for any $s \geq 0$, there exists $T > 0$ such that solutions converge in $C([0,T]; H^s)$ to those of the KdV equation as $\varepsilon \to 0$.
For the mKdV-B equation ($\alpha=3$), we establish analogous results--unconditional uniform well-posedness  and inviscid limit behavior in $H^s(\mathbb{T})$ for $s \geq 1/2$.
\end{abstract}


\maketitle

\tableofcontents

\section{Introduction}

In this paper, we are interested in the Korteweg-de Vries-Burgers(KdV-B) equation under the rough initial data on a torus:
	\begin{align}\label{00}
		\partial_t u(t,x)+\partial_x^3u(t,x)-\varepsilon\partial_x^2u(t,x)=\partial_x(u(t,x))^2, \quad
		u(0,x)=\phi,
	\end{align}
where $\varepsilon\in(0, 1]$ represents the diffusion coefficient, $\mathbb{T}=[0,2\pi]$, $u(t,x):\mathbb{R}^{+}\times\mathbb{T}\rightarrow \mathbb{R}$ is a real-valued function.
When $\varepsilon=0$, the equation \eqref{00} reduces to the well-known Korteweg-de Vries (KdV) equation
\begin{align}\label{1.4}
		\partial_t u(t,x)+\partial_x^3u(t,x)=\partial_x(u(t,x))^2, \quad u(0,x)=\phi.
\end{align}

It is known that the KdV-Burgers equation, either on a torus or on the whole space, is globally well-posed in $H^{s}$ with $s\geq-1$, i.e., there exists a unique solution in $C([0, T]; H^{s})$ for any initial value in $H^{s}$ and any positive time $T>0$;
see \cite{ X-M, L-F, J}.
Regarding the KdV equation, Bourgain \cite{J} established local well-posedness (LWP) in $L^2$ and introduced the $X^{s,b}$-space framework.  Kenig, Ponce and Vega \cite{CGL} developed the bilinear estimates and obtained that \eqref{1.4} is locally well-posed for $s > -1/2$. The sharp result on global well-posedness in $H^s$ was obtained in \cite{K}, it was shown that \eqref{1.4} is globally well-posed in $H^s$ for $s > -1$.
In fact, global well-posedness on both the torus and the whole space for $s \geq -1$ has been well established, referring to \cite{K, RM, JM}.

We also consider defocusing modified Korteweg-de Vries-Burgers
(mKdV-B) equation under the rough initial data on a torus
	\begin{align}\label{8-1}
		\partial_t u(t,x)+\partial_x^3u(t,x)-\varepsilon\partial_x^2u(t,x)=\partial_x(u(t,x))^3, \quad u(0,x)=\phi.
	\end{align}
If $\varepsilon=0$, \eqref{8-1} reduces to the modified KdV(mKdV) equation
\begin{equation}\label{8-2}
		\partial_t u(t,x)+\partial_x^3u(t,x)=\partial_x(u(t,x))^3, \quad u(0,x)=\phi.
\end{equation}
Bourgain \cite{J} proved local well-posedness of \eqref{8-2} in $H^s(\mathbb{T})$ for $s \geq 1/2$, which was later extended to global well-posedness by Colliander, Keel, Staffilani, Takaoka, and Tao \cite{JM}.
In fact, Kappeler and Topalov \cite{KT} proved the global well-posedness in $L^2$ by the inverse scattering method, which essentially depends on the complete integrability of the mKdV equation.  Further improvements on local well-posedness for  $3/8 < s < 1/2$ were made by Takaoka and Tsutsumi \cite{TT}, and very recently  improved by Nakanishi et al. \cite{N}, who pushed the local well-posedness to $H^s(\mathbb{T})$ for $s > 1/3$ (with local existence known for $s > 1/4$).

The main purpose of this paper is to study the inviscid limit of \eqref{00} and \eqref{8-1} on the torus in low-regularity Sobolev spaces $H^s(\mathbb{T})$.
Guo and Wang \cite{ZB} studied the inviscid limit behavior for the Cauchy problem of \eqref{00} on $\mathbb{R}$. They proved uniform global well-posedness in $H^{s}(\mathbb{R})$ for $s > -{3/4}$ and that the solution converges in $C([0, T]; H^s(\mathbb{R}))$ to that of  \eqref{1.4} for any $T > 0$, by using a $l^1$-variant $X^{s,b}$ space and I-method. Zhang and Han \cite{ZH} showed that \eqref{8-1} is uniformly globally
well-posed in $H^s$
$(s \geq 1)$ and  for any $s \geq 1$  and $T > 0$,
its solution converges in $C([0, T]; H^s(\mathbb{R}))$ to that of the \eqref{8-2} if $\varepsilon$ tends to $0$. However, the inviscid limit behavior of \eqref{00} and \eqref{8-1} on $\mathbb{T}$ in low-regularity spaces $H^s$ remains less explored.

We prove that the solution of \eqref{00} and \eqref{8-1} converges to that of \eqref{1.4} and \eqref{8-2} as $\varepsilon\to 0$ in $C([0, T]; H^s(\mathbb{T}))$ for $s \geq 0$ in the KdV case and $s \geq 1/2$ in the mKdV case. To achieve this, our primary focus is on establishing unconditional uniform well-posedness of \eqref{00} and \eqref{8-1} in the Sobolev space $H^{s}(\mathbb{T})$. This yields estimates that are uniform in $\varepsilon$ and independent of the dissipative terms.
In contrast to previous works, our analysis allows us to prove existence, uniqueness, and continuous dependence of solutions directly in $C([0,T]; H^s(\mathbb{T}))$, without relying on auxiliary function spaces such as Bourgain spaces or
weighted Sobolev spaces.
Recall the definition from Kato \cite{T},
unconditional uniqueness is a concept of uniqueness which does
not depend on how solutions are constructed. Babin et al. \cite{V} obtained unconditional uniqueness for the KdV equation in $L^2$, later adapted by Kwon-Oh \cite{ST} for the mKdV equation in $H^{1/2}(\mathbb{T})$.

For KdV-Burgers equation,
our first result establishes the uniform local well-posedness of \eqref{00} in $C([0,T]; H^s(\mathbb{T}))$ for $s\geq0$. We need to control the solution uniformly in $\varepsilon$, which is independent of the properties of dissipative term. We prove uniform well-posedness result using variable changes and differentiation-by-parts methods within a straightforward framework. By introducing the twisted variable $v(t,x): = e^{t(\partial_x^3-\varepsilon\partial_x^2)}u(t,x)$ to eliminate derivative losses in the quadratic nonlinearity, inspired by \cite{V}.

Building on this uniform well-posedness framework, our second result analyzes the inviscid limit of $\eqref{00}$ when $\varepsilon\rightarrow 0$. We consider the difference equation between \eqref{00} and \eqref{1.4}, treating the dissipative term as
the perturbation and then appyling the energy estimates. For the convenience, we use $S^{\varepsilon}$ and $S$ to denote the solution maps for
the KdV-Burgers and KdV, respectively for $T>0$.

Furthermore, we establish unconditional uniform well-posedness and the inviscid limit for the mKdV-Burgers equation in $H^{s}(\mathbb{T})$ for $s \geq 1/2$ as $\varepsilon \to 0$. We denote by ${S}_m^{\varepsilon}$ and ${S}_m$ the solution maps for mKdV-Burgers and mKdV, respectively.

Our main results are summarized as follows:

\begin{thm}\label{00-1}
 Let $\phi\in H^s(\mathbb{T})$ for $s\geq0$. Then for any $T>0$, there exists a unique solution $u$ of \eqref{00} satisfies for all $0<\varepsilon\leq 1$
$$\|u\|_{L_t^{\infty}H_x^s}\lesssim C(T,\|\phi\|_{H^s}),$$
and the solution map $S^{\varepsilon}: \phi\rightarrow u$ is continuous from $H^s$ to $C([0,T ],H^s)$ uniformly on $\varepsilon\in (0,1].$
\end{thm}
\begin{thm}\label{00-2}
 Let $\phi\in H^{s}(\mathbb{T})$, $s\geq 0$. For some $T > 0$, $t \in (0,T)$, then
 $$\lim_{\varepsilon\rightarrow 0} \|S^{\varepsilon}(\phi)-S(\phi)\|_{C([0,T],{H^s})}=0.$$
\end{thm}

\begin{thm}\label{8-6}
 Let $\phi\in H^s(\mathbb{T})$ and $s\geq1/2$. Then for some $T>0$, there exists a unique solution $u$ of \eqref{8-1} satisfies for all $0<\varepsilon\leq 1$
$$\|u\|_{L_t^{\infty}H_x^s}\lesssim C(T,\|\phi\|_{H^s}),$$
and the solution map $S_m^{\varepsilon}: \phi\rightarrow u$ is continuous from $H^s$ to $C([0,T ],H^s)$ uniformly on $\varepsilon\in (0,1].$
\end{thm}
\begin{thm}\label{8-5}
 Let $\phi\in H^{s}$, $s\geq 1/2$. For some $T > 0$, we have
 $$\lim_{\varepsilon\rightarrow 0} \|{S}_m^{\varepsilon}(\phi)-{S}_m(\phi)\|_{C([0,T],{H^s})}=0.$$
\end{thm}

The rest of the paper is organized as following. We present some notations in Section 2. We present uniform LWP and prove limit behavior for KdV-Burgers equation in Section 3 and Section 4. Theorem \ref{8-6} and Theorem \ref{8-5} are proved in Section 5 and Section 6.

\section{Notations and preliminaries}
In this section, we present some basic notations and some preliminaries to be frequently used in this paper.

\subsection{Basic notations}\mbox{}

We denote by $A \lesssim B$ or $B \gtrsim A$ the statement $A \leq C B$ for some constant $C > 0$, by $A \sim B$ the statement $A \lesssim B \lesssim A$, and by $A \ll B$ or $A \gg B$ the statement $A \leq C^{-1} B$ for some large constant $C\geq 10$.

We denote $\left<\cdot,\cdot \right>$ to be the inner product in $L^2(\mathbb{T})$, that is,
\begin{equation*}
\left< f,g \right>=\mbox{Re}\int_{\mathbb{T}}f(x) \overline{g(x)}dx \quad\text{and} \quad \|f\|_{L^2(\mathbb{T})}=\sqrt{\left< f,f \right>}.
\end{equation*}
The Fourier transform and the Fourier inverse formula of a function $f$ on $\mathbb{T}$ is defined by
\begin{align*}
\hat{f}_k=\mathcal{F}_k[f]&=\frac{1}{2\pi}\int_{\mathbb{T}}e^{-i{k}\cdot x}f(x)dx, \quad f(x)=\sum_{{k}\in\mathbb{Z}}e^{i{k}\cdot x}\hat{f}_{{k}}.
\end{align*}
The following standard properties of the Fourier transform hold:
\begin{align*}
\|f\|^2_{L^2}=2\pi\sum_{{k}\in\mathbb{Z}}|\hat{f}_{{k}}|^2=\|\hat{f}\|^2_{L^2}    &\qquad\text{(Plancherel identity)};\\
\widehat{(fg)}({k})=\sum_{{k}_1\in\mathbb{Z}}\hat{f}_{{k}-{k}_1}\hat{g}_{{k}_1}   &\qquad\text{(Convolution formula)};\\
\left< f,g \right>=2\pi\sum_{{k}\in\mathbb{Z}}\hat{f}_{{k}}\overline{\hat{g}_k}   &\qquad\text{(Parseval identity)}.
\end{align*}

 The Sobolev space $W^{s,p}(\mathbb{T})$ and the norm are defined by
\begin{align*}
W^{s,p}(\mathbb{T})=\Big\{f\in L^p(\mathbb{T}): D^{\alpha}f\in L^p(\mathbb{T}),\>|\alpha|\leq m\Big\},\quad \|f\|_{W^{s,p}}=\sum_{0\leq |\alpha|\leq s}\|D^{\alpha}f\|_{L^p}.\end{align*}
When $p=2$, $W^{s,p}(\mathbb{T})=H^s(\mathbb{T})$, which has the following equivalent norm
$$\|f\|_{H^{s}(\mathbb{T})}=\|J^{s}f\|_{L^2(\mathbb{T})}=\Big(2\pi\sum_{{k}\in\mathbb{Z}}(1+{|k|}^2)^{s}|\hat{f}_{{k}}|^2\Big)^{\frac{1}{2}},$$
$J$ is the Bessel potential operator defined as
$J=(1-\partial_{xx})^{\frac{1}{2}}.$ Note that for a mean-zero $L^2$ function $u$, $\|u\|_{H^s(\mathbb{T})}\approx \||k|^{s}|\hat{u}_{{k}}|\|_{L^2(\mathbb{T})}$.

We denote $1_{E}$ to be the characteristic function on a set $E$, and define the Littlewood-Paley projections  $\mathbb{P}_{\leq N}$ and $\mathbb{P}_{> N}$ as
 \begin{align}
 \mathbb{P}_{\leq N}f:=\mathcal{F}_k^{-1}(1_{|k|\leq N}\mathcal{F}_k [f]),\qquad
 \mathbb{P}_{> N}f:=\mathcal{F}_k^{-1}(1_{|k|> N}\mathcal{F}_k[f]).\label{1-12}
 \end{align}

We denote zero-mode projection $\mathbb{P}_0$  and nonzero-mode projection $\mathbb{P}$ by
\begin{align}\mathbb{P}_0 f=\frac{1}{2\pi}\int_{\mathbb{T}}f dx,       \qquad  \mathbb{P} f(x)= f(x)-\frac{1}{2\pi}\int_{\mathbb{T}}f dx.\label{22-1} \end{align}
The operator $\partial_x^{-1}$ is defined by
\begin{equation*}\label{def:px-1}
\widehat{\partial_x^{-1}f}({k})
=\Bigg\{ \aligned
    &(i{k})^{-1}\hat f_k,\quad &\mbox{when } {k}\ne 0,\\
    &0,\quad &\mbox{when } {k}= 0.
   \endaligned
\end{equation*}
Moreover, the following relation holds
$$\partial_x^{-1}\partial_x f=\partial_x\partial_x^{-1}f=\mathbb{P}f, \quad \text{if} \>\>\mathbb{P}_0 f=0.$$

\subsection{The basic tools}\mbox{}

We establish some estimates for the phase function
\begin{align}\tilde{\phi}: = k^3-k_1^3-k_2^3-k_3^3=3(k_1+k_2)(k_2+k_3)(k_1+k_3)\label{0-12}.\end{align}%+i\varepsilon(k^2-k_1^2-k_2^2-k_3^2).$$
We decompose the set $\{(k_1,k_2,k_3)\in\mathbb{Z}^3:k_1+k_2+k_3=k\} $ into the following two subsets:
\begin{align}
&\Gamma_0(k):=\{(k_1,k_2,k_3)\in\mathbb{Z}^3: k_1+k_2+k_3=k, \>(k_1+k_2)(k_1+k_3)(k_2+k_3)= 0\},\label{0-8}\\
&\Gamma(k)\>:=\{(k_1,k_2,k_3)\in\mathbb{Z}^3: k_1+k_2+k_3=k, \>(k_1+k_2)(k_1+k_3)(k_2+k_3)\neq0\}.\label{0.1}
\end{align}
$\Gamma(k)$ can be decomposed into $\Gamma(k)=\Gamma_1(k)\cup\Gamma_2(k)$, $k_m= \max(|k|,|k_1|,|k_2|,|k_3|)$,  where
\begin{align*}
&\Gamma_1(k):=\{(k_1,k_2,k_3)\in\Gamma(k):  |\tilde{\phi}|\ll\frac{1}{4}|k_m|^2\},\\
&\Gamma_2(k):=\{(k_1,k_2,k_3)\in\Gamma(k):  |\tilde{\phi}|\gtrsim\frac{1}{4}|k_m|^2\}.
\end{align*}

\begin{lem}[See \cite{BY}]\label{5}
Let $k \in \mathbb{Z}$. Then the following results hold.
\begin{enumerate}
    \item If $(k_1, k_2, k_3) \in \Gamma(k)$ then
    \begin{align*}
        |\tilde{\phi}| \gtrsim |k_m|.
    \end{align*}
    \item If $(k_1, k_2, k_3) \in \Gamma_1(k)$ then
    \begin{align*}
        |k| \sim |k_1| \sim |k_2| \sim |k_3|.
    \end{align*}
    \item $\Gamma_2(k)$ can be further decomposed into $\Gamma_2(k) = \Gamma_{21}(k) \cup \Gamma_{22}(k)$, with
\begin{align*}
\Gamma_{21}(k) &:= \left\{ (k_1, k_2, k_3) \in \Gamma(k) : \frac{1}{4} |k_m|^2 \lesssim |\phi| \ll |k_m|^{{15}/{7}} \right\},\\
\Gamma_{22}(k) &:= \left\{ (k_1, k_2, k_3) \in \Gamma(k) : |\phi| \gtrsim |k_m|^{{15}/{7}} \right\}.
\end{align*}
Moreover, for $(k_1, k_2, k_3) \in \Gamma_{21}(k)$ there exists $j, h \in \{1, 2, 3\}$ such that $|k_j + k_h| \ll |k_m|^{{5}/{7}}$.
\end{enumerate}
\end{lem}

\begin{lem}\label{0-5}
Let \( p(k) = k^3 + i\varepsilon k^2 \). We define the following functions:
\begin{align*}
Q_1(k, k_1, k_2) &= p(k) - p(k_1) - p(k_2), \quad \text{where } k = k_1 + k_2, \\
Q_2(k, k_1, k_2, k_3) &= Q_1(k, k_1 + k_2, k_3) + Q_1(k_1 + k_2, k_1, k_2), \quad \text{where } k = k_1 + k_2 + k_3.
\end{align*}
Then the following inequalities hold:
\begin{equation*}
\begin{aligned}
|Q_1(k, k_1, k_2)| &\gtrsim |k k_1 k_2|,\\
|Q_2(k, k_1, k_2, k_3)| &\gtrsim |\tilde{\phi}| \gtrsim |(k_1 + k_3)(k_2 + k_3)(k_1 + k_2)|.
\end{aligned}
\end{equation*}
\end{lem}
\begin{proof}
From the expression of $Q_1(k,k_1,k_2)$, we have
\begin{align*}
& |Q_1(k,k_1,k_2)|=|3kk_1k_2+i\varepsilon(k^2-k_1^2-k_2^2)|\gtrsim|kk_1k_2|.\end{align*}
For $Q_2(k,k_1,k_2,k_3)=Q_1(k,k_1+k_2,k_3)+Q_1(k_1+k_2,k_1,k_2)$, we have
     \begin{align}
     Q_2(k,k_1,k_2,k_3)&=3(kk_3+k_1k_2)(k_1+k_2)+i\varepsilon(k^2-k_1^2-k_2^2-k_3^2)\nonumber\\
     & =3(k_1+k_3)(k_1+k_2)(k_2+k_3)+i\varepsilon(k^2-k_1^2-k_2^2-k_3^2).\label{0.7}\end{align}
Hence, we obtain$$  |Q_2(k,k_1,k_2,k_3)|   \gtrsim|3(k_1+k_3)(k_2+k_3)(k_1+k_2)|=3|\tilde{\phi}|.$$
Therefore, the lemma is proved.
\end{proof}

\begin{lem}\label{0-6}
Let \( \tilde{\phi} \) be defined in \eqref{0-12}, and let \( k = k_1 + k_2 + k_3 \) with \( k_1 + k_2 \neq 0 \) and $k\neq 0$, \( C_1(v_1, v_2, v_3) \),
\( C_2(v_1, v_2, v_3) \) and \( C_3(v_1, v_2, v_3) \) satisfy the following inequalities:
\begin{align*}
\left|\mathcal{F}_k[C_1(v_1, v_2, v_3)]\right| &\lesssim \sum_{\Gamma_0(k), \, k_3 \neq 0} |k_3|^{-1} \left|\mathcal{F}_{k_1}[v_1(t)] \mathcal{F}_{k_2}[v_2(t)] \mathcal{F}_{k_3}[v_3(t)]\right|, \\
\left|\mathcal{F}_k[C_2(v_1, v_2, v_3)]\right| &\lesssim \sum_{\Gamma(k), \, k_3 \neq 0} |k_3|^{-1} |\tilde{\phi}|^{-1} \left|\mathcal{F}_{k_1}[v_1(t)] \mathcal{F}_{k_2}[v_2(t)] \mathcal{F}_{k_3}[v_3(t)]\right|,\\
\left|\mathcal{F}_k[C_3(v_1, v_2, v_3)]\right| &\lesssim \sum_{\Gamma(k)} |k|  |\tilde{\phi}|^{-1}\left|\mathcal{F}_{k_1}[v_1(t)] \mathcal{F}_{k_2}[v_2(t)] \mathcal{F}_{k_3}[v_3(t)]\right|,
\end{align*}
where \( \Gamma_0(k) \) and \( \Gamma(k) \) are defined in \eqref{0-8} and \eqref{0.1}, respectively. We assume $|k_m|=|k_1|>N$, then the following estimates hold:

(1)  Let $s > -1/2$, if $v_1, v_2, v_3 \in H^s$,  then the function $C_1(v_1, v_2, v_3)\in H^s$ and
\[
\| C_1(v_1, v_2, v_3) \|_{H^s} \lesssim \| v_1 \|_{L_t^\infty H_x^s} \| v_2 \|_{L_t^\infty H_x^s} \| v_3 \|_{L_t^\infty H_x^s}.
\]

(2)  Let $s\geq 0$, if $v_1\in H^{s-1}$, and $v_2, v_3 \in H^s$,  then the function $C_2(v_1, v_2, v_3)\in H^s$ and
\[
\| C_2(v_1, v_2, v_3) \|_{H^s} \lesssim \| v_1 \|_{L_t^\infty H_x^{-1+s}} \| v_2 \|_{L_t^\infty H_x^s} \| v_3 \|_{L_t^\infty H_x^s}.
\]

(3) Let $s\geq {1}/{2}$, if $v_1\in H^{s-1/2}$, and $v_2, v_3 \in H^s$,  then $C_3(v_1, v_2, v_3)\in H^s$ and
\[
\| C_3(v_1, v_2, v_3) \|_{H^s} \lesssim \| v_1 \|_{L_t^\infty H_x^{-1/2+s}} \| v_2 \|_{L_t^\infty H_x^s} \| v_3 \|_{L_t^\infty H_x^s}.
\]
\end{lem}

\begin{proof}
For $(k_1,k_2,k_3)\in \Gamma_0(k)$, since $k_1+k_2\neq0$, so either $k_1+k_3=0$ or $k_2+k_3=0$. We assume $k_1+k_3=0$, then $k=k_2$, it follows that
\begin{align*}|\mathcal{F}_{k}[C_1(v_1,v_2,v_3)]|
\lesssim \sum_{k_3\neq 0}|k_3|^{-1}\mathcal{F}_{-k_3}[v_1(t)]\mathcal{F}_{k}[v_2(t)]\mathcal{F}_{k_3}[v_3(t)].
\end{align*}
Taking $H^s$ norm, we obtain
\begin{align*}
\left\| {C}_{1}(v_1, v_2, v_3) \right\|_{H^s} &\lesssim \Big\|\sum_{k_3\neq 0}|k|^s\mathcal{F}_{k}[v_2(t)]|k_3|^{-1}\mathcal{F}_{-k_3}[v_1(t)]\mathcal{F}_{k_3}[v_3(t)]\Big\|_{L^2}\\
&\lesssim \big\|\mathcal{F}_{k}[v_2(t)]\big\|_{H^s}\Big\|\sum_{k_3\neq 0}|k_3|^{-1}|\mathcal{F}_{-k_3}[v_1(t)]|\left|\mathcal{F}_{k_3}[v_3(t)]\right|\Big\|_{L^{\infty}}\\
&\lesssim  \|v_2(t)\|_{H^s}\Big\|\sum_{k_3\neq 0}|k_3|^{-s-1}\mathcal{F}_{-k_3}[v_1(t)]|k_3|^s\mathcal{F}_{k_3}[v_3(t)]\Big\|_{L^{\infty}}\\
&\lesssim   \big\||k|^{-s-1}\mathcal{F}_{k}[v_1(t)]\big\|_{L^2}\|v_2(t)\|_{H^s} \|v_3(t)\|_{H^s}.
\end{align*}
For $s>-{1}/{2}$, this implies that
$$\left\| {C}_{1}(v_1, v_2, v_3) \right\|_{H^s}\lesssim\|v_1\|_{L_t^{\infty}H_x^s}\|v_2\|_{L_t^{\infty}H_x^s}\|v_3\|_{L_t^{\infty}H_x^s}.$$

Since $\Gamma(k)=\Gamma_1(k)\cup\Gamma_2(k)$, we decompose the inequality into two parts:
\begin{align*}
\left|\mathcal{F}_k[C_2(v_1, v_2, v_3)]\right|
&\lesssim \sum_{\Gamma_1(k), \, k_3 \neq 0} |k_3|^{-1} |\tilde{\phi}|^{-1} \big|\mathcal{F}_{k_1}[v_1(t)] \mathcal{F}_{k_2}[v_2(t)] \mathcal{F}_{k_3}[v_3(t)]\big|\\
&\qquad+\sum_{\Gamma_2(k), \, k_3 \neq 0} |k_3|^{-1} |\tilde{\phi}|^{-1} \big|\mathcal{F}_{k_1}[v_1(t)] \mathcal{F}_{k_2}[v_2(t)] \mathcal{F}_{k_3}[v_3(t)]\big|\\
&=:|\mathcal{F}_{k} [C_{21}(v_1,v_2,v_3)]|+|\mathcal{F}_{k} [C_{22}(v_1,v_2,v_3)]|.
\end{align*}

Let
$
\hat{v}_{j, k_j}=\mathcal{F}_{k_j}[v_j],
$
 $j=1,2,3,$
and assume that \( \hat{v}_{j, k_j} \) are nonnegative for any \( t \in [0,T] \), otherwise one may replace them by \( |\hat{v}_{j, k_j}| \).
By Lemma \ref{5}, for $(k_1,k_2,k_3)\in\Gamma_1(k)$, we have  $|\tilde{\phi}| \gtrsim |k_m|$ and $|k|\sim|k_1|\sim|k_2|\sim|k_3|$.

Therefore,
\begin{align*}
|\mathcal{F}_{k} [C_{21}(v_1,v_2,v_3)]|
&\lesssim  \sum_{|k_1|\sim|k_2|\sim|k_3|\sim|k|}|k|^{-1}|k_m|^{-1}
          \hat{v}_{1, k_1} \hat{v}_{2, k_2} \hat{v}_{3, k_3}\\
&\lesssim \sum_{|k_1|\sim|k_2|\sim|k_3|\sim|k|}|k_m|^{-1}|k_1|^{-\frac{1}{3}}\hat{v}_{1, k_1}
|k_2|^{-\frac{1}{3}}\hat{v}_{2, k_2}|k_3|^{-\frac{1}{3}}\hat{v}_{3, k_3}.
\end{align*}
For $s\geq 0$, we also get $|k|^s\lesssim|k|^{3s}\sim |k_1|^s|k_2|^s|k_3|^s$. Then we have
\begin{align*}
\|\mathcal{F}_{k} [C_{21}(v_1,v_2,v_3)]\|_{H^s}
&\lesssim \Big\|\sum_{k_1+k_2+k_3=k}|k|^s|k|^{-1}|k_m|^{-1}\hat{v}_{1, k_1} \hat{v}_{2, k_2} \hat{v}_{3, k_3}\Big\|_{L^2}\\
&\lesssim \Big\|\sum_{|k_1|\sim|k_2|\sim|k_3|}|k_m|^{-1}|k_1|^{s-\frac{1}{3}}\hat{v}_{1, k_1}
\cdot|k_2|^{s-\frac{1}{3}}\hat{v}_{2, k_2}
\cdot|k_3|^{s-\frac{1}{3}}\hat{v}_{3, k_3} \Big\|_{L^2}\\
&\lesssim \big\|\partial_x^{-1}\partial_x^{s-\frac{1}{3}}v_1\big\|_{L_t^{\infty}L_x^6}
\big\|\partial_x^{s-\frac{1}{3}}v_2\big\|_{L_t^{\infty}L_x^6}\big\|\partial_x^{s-\frac{1}{3}}v_3\big\|_{L_t^{\infty}L_x^6}\\
&\lesssim \|v_1\|_{L_t^{\infty}H_x^{s-1}}
\|v_2\|_{L_t^{\infty}H_x^s}\|v_3\|_{L_t^{\infty}H_x^s}.
\end{align*}
For the last two line, we obtain the inequality by choosing $|k_m|^{-1}=|k_1|^{-1}$.

By Lemma \ref{5}, for $(k_1,k_2,k_3)\in\Gamma_2(k)$, without loss of generality, we assume $|k_m|=|k_1|$, we have $|\tilde{\phi}|\gtrsim|k_m|^2=|k_1|^2$, which leads to the following estimate by the Cauchy-Schwarz inequality
\begin{align*}
|\mathcal{F}_{k}[{C}_{22}(v_1, v_2, v_3)]|&\lesssim
\Big(\sum_{k_2,k_3}|k_m|^{-2}|k_3|^{-2}\Big)^{\frac{1}{2}}\Big(\sum_{k_2,k_3}|k_m|^{-2}|\hat{v}_{1, k-k_2-k_3}|^2 |\hat{v}_{2, k_2}|^2 |\hat{v}_{3, k_3}|^2\Big)^{\frac{1}{2}}\\
&\lesssim \Big(\sum_{k_2,k_3}|k_m|^{-2}|\hat{v}_{1, k-k_2-k_3}|^2 |\hat{v}_{2, k_2}|^2 |\hat{v}_{3, k_3}|^2\Big)^{\frac{1}{2}}.
\end{align*}
By taking square of the inequality above and summing up the results for $k\in\mathbb{Z}$, we obtain
\begin{align}
\left\| {C}_{22}(v_1, v_2, v_3) \right\|_{H^\sigma}^2 &\lesssim \sum_{k \neq 0} |k|^{2\sigma} \sum_{k_1 + k_2 + k_3 = k} |k_m|^{-2} \left| \hat{v}_{1, k_1}\right|^2 \left| \hat{v}_{2, k_2} \right|^2 \left|\hat{v}_{3, k_3} \right|^2 \nonumber\\
&\lesssim \sum_{k_1, k_2, k_3} \left| \langle k_1 \rangle^{\sigma_1} \mathcal{F}_{k_1}[v_1(t)] \right|^2 \left| \langle k_2 \rangle^{\sigma_2} \mathcal{F}_{k_2}[v_2(t)] \right|^2 \left| \langle k_3 \rangle^{\sigma_3} \mathcal{F}_{k_3}[v_3(t)] \right|^2 \nonumber\\
&\lesssim \| v_1 \|_{L_t^\infty H_x^{\sigma_1}}^2 \| v_2 \|_{L_t^\infty H_x^{\sigma_2}}^2 \| v_3 \|_{L_t^\infty H_x^{\sigma_3}}^2,\label{1-11}
\end{align}
where $0\leq\sigma\leq 1$ and $-1\leq \sigma_j\leq 0, j=1, 2,3$ satisfying $\sigma_1+\sigma_2+\sigma_3=\sigma-1.$

In particular, by choosing $\sigma_1=-1-s$, and $\sigma_2=\sigma_3=\sigma=s$ in \eqref{1-11}, for $s\geq 0$, we have $\sigma_1\leq -1\leq s-1$, then we obtain the following result:
\begin{align*}
\|C_{22}(v_1,v_2,v_3)\|_{H^s}
&\lesssim\|v_1\|_{L_t^{\infty}H_x^{-1+s}}\|v_2\|_{{L_t^{\infty}H_x^s}}\|v_3\|_{{L_t^{\infty}H_x^s}}.
\end{align*}

Similarly, we decompose $\left|\mathcal{F}_k[C_{3}(v_1,v_2,v_3)]\right|$ into two parts:
\begin{align*}
\left|\mathcal{F}_k[C_{3}(v_1,v_2,v_3)]\right|
&\lesssim \sum_{\Gamma_1(k)} |k||\tilde{\phi}|^{-1} \big|\mathcal{F}_{k_1}[v_1(t)] \mathcal{F}_{k_2}[v_2(t)] \mathcal{F}_{k_3}[v_3(t)]\big|\\
&\qquad+\sum_{\Gamma_2(k)} |k| |k_m|^{-2} \big|\mathcal{F}_{k_1}[v_1(t)] \mathcal{F}_{k_2}[v_2(t)] \mathcal{F}_{k_3}[v_3(t)]\big|\\
&=:|\mathcal{F}_{k} [C_{31}(v_1,v_2,v_3)]|+|\mathcal{F}_{k} [C_{32}(v_1,v_2,v_3)]|.
\end{align*}
For $(k_1,k_2,k_3)\in\Gamma_1(k)$, the following inequality also holds in this case: \( |\tilde{\phi}| \gtrsim |k_m|\Lambda \), where
\begin{align}\label{99-7}
\Lambda = \Lambda_k := \min\left( |k_1 + k_2||k_2 + k_3|, |k_2 + k_3||k_3 + k_1|, |k_3 + k_1||k_1 + k_2| \right).
\end{align}
$\tilde{\phi}$ satisfies  \eqref{0-12}. For $\Lambda$ as in \eqref{99-7}, we have $\max(|k_2|, |k_3|, \Lambda) \gtrsim |k_1|$, since we have $|k_2| \sim |k_1|$ or $|k_3| \sim |k_1|$ if $\Lambda \ll |k_1|$.

 Using the Cauchy-Schwarz inequality, for $\delta>0$, we have
 \begin{align*}
|\mathcal{F}_{k} [C_{31}(v_1,v_2,v_3)]|&\lesssim\sum_{\Gamma_1(k)}|\Lambda|^{-1}\hat{v}_{1,k_1}\hat{v}_{2,k_2}\hat{v}_{3,k_3}\\
&\lesssim|k|^{-\frac{1}{2}}\Big( \sum_{k_2,k_3}|\Lambda|^{-1-\delta}\Big)^{\frac{1}{2}}\Big( \sum_{k_2,k_3}|\Lambda|^{-1+\delta}|k||\hat{v}_{1,k-k_2-k_3}|^2|\hat{v}_{2,k_2}|^2|\hat{v}_{3,k_3}|^2
\Big)^{\frac{1}{2}}\\
&\lesssim \Big( \sum_{k_1+k_2+k_3=k}|\Lambda|^{-1+\delta}|k||\hat{v}_{1,k_1}|^2|\hat{v}_{2,k_2}|^2|\hat{v}_{3,k_3}|^2\Big)^{\frac{1}{2}},\\
|\mathcal{F}_{k} [C_{32}(v_1,v_2,v_3)]|&\lesssim\sum_{\Gamma_2(k)}|k_m|^{-1}\hat{v}_{1,k_1}\hat{v}_{2,k_2}\hat{v}_{3,k_3}\\
&\lesssim\Big( \sum_{k_2,k_3}\frac{1}{|k_2||k_3||k_m|}\Big)^{\frac{1}{2}}\Big( \sum_{k_2,k_3}\frac{|k_2||k_3|}{|k_m|}|\hat{v}_{1,k-k_2-k_3}|^2|\hat{v}_{2,k_2}|^2|\hat{v}_{3,k_3}|^2
\Big)^{\frac{1}{2}}\\
&\lesssim \Big( \sum_{k_1+k_2+k_3=k}|k_2||k_3|{|k_m|}^{-1}|\hat{v}_{1,k_1}|^2|\hat{v}_{2,k_2}|^2|\hat{v}_{3,k_3}|^2\Big)^{\frac{1}{2}}.
\end{align*}
Assume $|k_m|=|k_1|$, for $s\geq1/2$, we have $|k|\lesssim|k|^{2s}$, and $|\Lambda|^{-1+\delta}\lesssim 1$, which further yields that
\begin{align*}
\|C_{31}(v_1,v_2,v_3)\|^2_{H^s}
&\lesssim \sum_{k \neq 0}|k|^{2s}\sum_{\Gamma_1(k)}|k||\hat{v}_{1,k_1}|^2|\hat{v}_{2,k_2}|^2|\hat{v}_{3,k_3}|^2\nonumber\\
&\lesssim \sum_{k_1, k_2, k_3}|k|^{2s+1}|\hat{v}_{1,k_1}|^2|\hat{v}_{2,k_2}|^2|\hat{v}_{3,k_3}|^2\\
\nonumber
%&\\
&\lesssim\sum_{k_1, k_2, k_3} \left| \langle k_1 \rangle^{\sigma_1} \mathcal{F}_{k_1}[v_1(t)] \right|^2 \left| \langle k_2 \rangle^{\sigma_2} \mathcal{F}_{k_2}[v_2(t)] \right|^2 \left| \langle k_3 \rangle^{\sigma_3} \mathcal{F}_{k_3}[v_3(t)] \right|^2 \nonumber\\
&\lesssim \| v_1 \|_{L_t^\infty H_x^{\sigma_1}}^2 \| v_2 \|_{L_t^\infty H_x^{\sigma_2}}^2 \| v_3 \|_{L_t^\infty H_x^{\sigma_3}}^2,
\end{align*}
where $\sigma_1, \sigma_2, \sigma_3$ satisfy $\sigma_1+\sigma_2+\sigma_3=s+1/2.$
By choosing $\sigma_1=s-\alpha$, $\sigma_2=\sigma_3=s$ in above inequality, for $s\geq 1/2$, and $\alpha\geq 1/2$. then we obtain the following result:
$$\|C_{31}(v_1,v_2,v_3)\|^2_{H^s}\lesssim \| v_1 \|_{L_t^\infty H_x^{s-1/2}}^2 \| v_2 \|_{L_t^\infty H_x^{s}}^2 \| v_3 \|_{L_t^\infty H_x^{s}}^2.$$
The second term we obtain
\begin{align*}
\|C_{32}(v_1,v_2,v_3)\|^2_{H^s}
&\lesssim \sum_{k \neq 0}|k|^{2s}\sum_{\Gamma_2(k)}|k_m|^{-1}|k_2||k_3||\hat{v}_{1,k_1}|^2|\hat{v}_{2,k_2}|^2|\hat{v}_{3,k_3}|^2\nonumber\\
&\lesssim \sum_{k_1, k_2, k_3} |k_m|^{-1}|k_1|^{2s}|\hat{v}_{1,k_1}|^2|k_2|^{2s}|\hat{v}_{2,k_2}|^2|k_3|^{2s}|\hat{v}_{3,k_3}|^2\\
\nonumber
&\lesssim \| v_1 \|_{L_t^\infty H_x^{s-1/2}}^2 \| v_2 \|_{L_t^\infty H_x^{s}}^2 \| v_3 \|_{L_t^\infty H_x^{s}}^2.
\end{align*}
Hence, the third result of Lemma holds for $s\geq 1/2$.
\end{proof}
Kato-Ponce inequality which was originally proved in \cite{TG} and extended to the endpoint case in \cite{ Li-KatoPonce} \cite{Li-D} recently.
\begin{lem}[Kato Ponce inequality]\label{7-1}
For $s>0$, $1<p\leq\infty$ and $1<p_1, p_2, p_3, p_4\leq\infty$ satisfying $\frac{1}{p}=\frac{1}{p_1}+\frac{1}{p_2}$ and $\frac{1}{p}=\frac{1}{p_3}+\frac{1}{p_4}$, the following inequality holds:
\begin{align*}
\|J^s(fg)\|_{L^p}\leq C(\|J^{s}f\|_{L^{p_1}}\|g\|_{L^{p_2}}+\|J^{s}g\|_{L^{p_3}}\|f\|_{L^{p_4}}),
\end{align*}
where the constant $C>0$ depends on $s, p_1, p_2, p_3, p_4$. In particular, when $s>\frac{1}{p}$, then
\begin{align*}
\|J^s(fg)\|_{L^p}\leq C\|J^{s}f\|_{L^{p}}\|J^{s}g\|_{L^{p}},
\end{align*}
where the constant $C>0$ depends on $s, p$.
\end{lem}
Moreover, we will need the following estimate for the commutator $[A, B] := AB - BA$, which was originally proved in \cite{WZ}.
\begin{lem}
Let $s \geq 0$, $s_1 > {1}/{2}$; then the following inequalities hold:
\begin{enumerate}\label{99-1}
    \item[(1)] For any $f \in H^s \cap H^{s_1}$ and $g \in H^{s+1} \cap H^{s_1+1}$:
    \[
    \left\| [J^s \partial_x, g] f \right\|_{L^2} \lesssim \|f\|_{H^s} \|g\|_{H^{s_1+1}} + \|f\|_{H^{s_1}} \|g\|_{H^{s+1}}.
    \]

    \item[(2)] For any $f \in H^s$ and $g \in H^{s+s_1+1}$:
    \[
    \left\| [J^s, g\partial_x] f \right\|_{L^2} \lesssim \|f\|_{H^s} \|g\|_{H^{s_{1}+1}} + \|f\|_{L^2} \|g\|_{H^{s+s_1+1}}.
    \]
\end{enumerate}
\end{lem}

\section{Uniform local well-posedness for the KdV-B equation}
We present below an alternative method for proving local and global well-posedness for $H^s$ data on $\mathbb{T}$. The method can be summarized as changing variables and differentiating by
parts in the time variable.

Without loss of generality, we can assume that the integral of the initial value $u_0$ is zero,
i.e. $\mathbb{P}_0u_0=0$
and $\mathbb{P}u_0 = u_0$,
where $\mathbb{P}_0$ and $\mathbb{P}$ are the zero-mode and nonzero-mode projection
operators defined in \eqref{22-1}. Since $\mathbb{P}_0u_0=0$, the conservation law
$\int_{\mathbb{T}}u(t,x)dx=\int_{\mathbb{T}}u_0(x) dx$
of KdV-Burgers equation implies that
$\mathbb{P}_0u(t,\cdot)=0$ for all $t\geq 0$.

By introducing the twisted variable $v(t,x)$
$$v(t,x): = e^{t(\partial_x^3-\varepsilon\partial_x^2)}u(t,x), \quad t\geq 0, x\in \mathbb{T},$$
the KdV-Burgers equation in \eqref{00} can be written as
\begin{equation*}
\partial_t v(t,x)=e^{t(\partial_x^3-\varepsilon\partial_x^2)}\partial_x\left(e^{-t(\partial_x^3-\varepsilon\partial_x^2)}v(t,x)\right)^2,\quad t\geq 0, x\in\mathbb{T}.
\end{equation*}

Using the Fourier transform, we can have
\begin{align}
\partial_t\hat{v}_k={ik}\sum_{k_1+k_2=k}e^{-isQ_1(k,k_1,k_2)}\hat{v}_{k_1}\hat{v}_{k_2},\label{1-2}
\end{align}
where $Q_1(k,k_1,k_2)=3kk_1k_2+i\varepsilon(k^2-k_1^2-k_2^2)$.

Integrating both sides from 0 to $t$, we have
\begin{align}
\hat{v}_k(t)=\hat{v}_k(0)+\int_0^t{ik}\sum_{k_1+k_2=k}e^{-i\tau Q_1(k,k_1,k_2)}\hat{v}_{k_1}\hat{v}_{k_2}d\tau.\label{1-5}
\end{align}
 Therefore any given solution is also
a solution of the integral equation \eqref{1-5} below.
Indeed, it suffices to check that for each $k$, one can change the order of sum and differentiation.
Since
$e^{-itQ_1(k,k_1,k_2)}=\partial_t\left(\frac{ie^{-itQ_1(k,k_1,k_2)}}{Q_1(k,k_1,k_2)}\right),$
we consider the integration by parts and rewrite $\eqref{1-5}$ as
\begin{align}
\hat{v}_k(t)&=\hat{v}_k(0)-\sum_{k_1+k_2=k}\frac{ke^{-i\tau Q_1(k,k_1,k_2)}}{Q_1(k,k_1,k_2)}\hat{v}_{k_1}\hat{v}_{k_2}\Big|_{0}^t\nonumber
\\ &\quad+\int_0^t\sum_{k_1+k_2=k}\frac{ke^{-i\tau Q_1(k,k_1,k_2)}}{Q_1(k,k_1,k_2)}
\big(\partial_\tau\hat{v}_{k_1}\hat{v}_{k_2}+\partial_{\tau}\hat{v}_{k_2}\hat{v}_{k_1}\big)d\tau.\label{1-3}
\end{align}
Since $\mathbb{P}_0(v) = 0$, the terms corresponding to $k_1=0$ or $k_2=0$ are not actually present in the above sums. The last two terms are symmetric with respect to $k_1$ and $k_2$ and thus we can consider only one of them.

Substituting \eqref{1-2} into $\partial_t\hat{v}_{k_1}$, then the first term of \eqref{1-3} yields that
\begin{align*}
&\quad \>k\sum_{k_{1}+k_{2}=k}\frac{e^{-itQ_1(k,k_1,k_2)}}{Q_1(k,k_1,k_2)}\partial_t\hat{v}_{k_1}\hat{v}_{k_2}\\
& = k\sum_{k_{1}+k_{2}=k}\frac{e^{-itQ_1(k,k_1,k_2)}}{Q_1(k,k_1,k_2)}\Big(\sum_{\mu+\lambda=k_1}ik_1 e^{-itQ_1(k_1,\mu,\lambda)}\hat{v}_{\mu}\hat{v}_{\lambda}\Big)\hat{v}_{k_2}\\
&
=
\sum_{\mu+\lambda+k_{2}=k}\frac{ik(\mu+\lambda)}{Q_1(k,\mu+\lambda,k_2)} e^{-itQ_2(k,\mu,\lambda,k_2)}\hat{v}_{\mu}\hat{v}_{\lambda}\hat{v}_{k_2}\\
&=\sum_{k_1+k_2+k_3=k} \frac{{ik(k_1+k_2)}e^{-itQ_2(k,k_1,k_2,k_3)}}{Q_1(k,k_1+k_2,k_3)}\hat{v}_{k_1}\hat{v}_{k_2}\hat{v}_{k_3},
\end{align*}
where
$Q_2(k,\mu,\lambda,k_2)=Q_1(k,\mu+\lambda,k_2)+Q_1(\mu+\lambda,\mu,\lambda)$.
Now, relabel the variables in the summation as:
$\mu \to k_1$,  $\lambda \to k_2$,  $k_2 \to k_3$.

Formula \label{1-4} can be expressed alternatively as
\begin{align}\hat{v}_k(t)=\hat{v}_k(0)-\widehat{A}_2(v,v)_k(t)+\widehat{A}_2(v,v)_k(0)+2\int_0^t\widehat{B}_{3}(v,v,v)_{k}(\tau)d\tau,\label{1-7}
\end{align}
where
\begin{align}
\widehat{A}_2(u,v)_{k}= \sum_{k_{1}+k_{2}=k}\frac{k e^{-itQ_1(k,k_1,k_2)}}{Q_1(k,k_1,k_2)}
\hat{u}_{k_1}\hat{v}_{k_2},
\end{align}
$$\widehat{B}_{3}(u,v,w)_{k}=\sum_{k_{1}+k_{2}+k_{3}=k}\frac{ik (k_1+k_2) e^{-itQ_2(k,k_1,k_2,k_3)}}{Q_1(k,k_1+k_2,k_3)}\hat{u}_{k_1}\hat{v}_{k_2}\hat{w}_{k_3}.$$

For $\widehat{B}_{3}(v,v,v)_{k}$, we decompose it into two parts, resonant term $\widehat{R}^0_{3}(v,v,v)_{k}$ and non-resonant term $\widehat{R}_{3}(v,v,v)_{k}$, corresponding to the sets $\Gamma_0(k)$ and $\Gamma(k)$ respectively, where $\Gamma_0(k)$ and $\Gamma(k)$ defined in \eqref{0-8} and \eqref{0.1}.
Then we have
\begin{align}
\widehat{B}_{3}(u,v,w)_{k}&=\sum_{(k_1,k_2,k_3)\in\Gamma_0(k)}\frac{ik (k_1+k_2) e^{-itQ_2(k,k_1,k_2,k_3)}}{Q_1(k,k_1+k_2,k_3)}\hat{u}_{k_1}\hat{v}_{k_2}\hat{w}_{k_3}\\
\nonumber
&\quad+\sum_{(k_1,k_2,k_3)\in\Gamma(k)}\frac{ik (k_1+k_2) e^{-itQ_2(k,k_1,k_2,k_3)}}{Q_1(k,k_1+k_2,k_3)}\hat{u}_{k_1}\hat{v}_{k_2}\hat{w}_{k_3}\nonumber\\
&=:\widehat{R}^0_{3}(v,v,v)_{k}+\widehat{R}_{3}(v,v,v)_{k}.\label{99-9}
\end{align}
For the term $\widehat{R}^0_{3}(v,v,v)_{k}$, we consider configurations where $(k_1+k_2)(k_1+k_3)(k_2+k_3)=0.$
Since $k_1+k_2\neq 0,$ we have $k_1+k_3=0$ or $k_2+k_3=0$.  In this case, the phase function simplifies to $Q_2(k,k_1,k_2,k_3)=i\varepsilon(k^2-k_1^2-k_2^2-k_3^2),$ then we have the following expression:
\begin{align}
\widehat{R}^0_{3}(v,v,v)_{k}&=\sum_{\Gamma_0(k)}\frac{ik (k_1+k_2)e^{t\varepsilon(k^2-k_1^2-k_2^2-k_3^2)}}{Q_1(k,k_1+k_2,k_3)}\hat{v}_{k_1}\hat{v}_{k_2}\hat{v}_{k_3}\nonumber\\
&=\sum_{\Gamma_0(k)}\frac{ik (k_1+k_2) e^{t\varepsilon k^2}}{Q_1(k,k_1+k_2,k_3)}{\mathcal{F}_{k_1}[e^{t\varepsilon\partial_x^2}v] \mathcal{F}_{k_2}[e^{t\varepsilon\partial_x^2}v]  \mathcal{F}_{k_3}[e^{t\varepsilon\partial_x^2}v]}. \label{0.2}
\end{align}%
Next, for the term $\widehat{R}_{3}(v,v,v)_{k}$, we consider configurations where
$(k_1+k_2)(k_1+k_3)(k_2+k_3)\neq 0.$ We can differentiate by parts one more time, i.e.,
$e^{-iQ_2(k,k_1,k_2,k_3)t}
=\partial_{t}\left(\frac{ie^{-itQ_2(k,k_1,k_2,k_3)}}{Q_2(k,k_1,k_2,k_3)}\right).$
Then
\begin{align*}
&\widehat{R}_{3}(v,v,v)_{k}\\
&=-\partial_t \Big(\sum_{\Gamma(k)}\frac{k (k_1+k_2) e^{-itQ_2(k,k_1,k_2,k_3)}}{Q_1(k,k_1+k_2,k_3)Q_2(k,k_1,k_2,k_3)}\hat{v}_{k_1}\hat{v}_{k_2}\hat{v}_{k_3}\Big)\\
&\quad+\sum_{\Gamma(k)}\frac{k(k_1+k_2)e^{-itQ_2(k,k_1,k_2,k_3)}}{Q_1(k,k_1+k_2,k_3)Q_2(k,k_1,k_2,k_3)}
\big(\partial_t\hat{v}_{k_1}\hat{v}_{k_2}\hat{v}_{k_3}+\partial_t\hat{v}_{k_2}\hat{v}_{k_1}\hat{v}_{k_3}+\partial_t\hat{v}_{k_3}\hat{v}_{k_1}\hat{v}_{k_2}\big).
\end{align*}
The term corresponding to $\partial_t\hat{v}_{k_1}$ is
\begin{align}
&\>\sum_{k_{1}+k_{2}+k_{3}=k}\frac{k(k_1+k_2)e^{-itQ_2(k,k_1,k_2,k_3)}}{Q_1(k,k_1+k_2,k_3)Q_2(k,k_1,k_2,k_3)}
\partial_t\hat{v}_{k_1}\hat{v}_{k_2}\hat{v}_{k_3}\nonumber\\
&=\sum_{k_{1}+k_{2}+k_{3}=k}\frac{k(k_1+k_2)
e^{-itQ_2(k,k_1,k_2,k_3)}}{Q_1(k,k_1+k_2,k_3)Q_2(k,k_1,k_2,k_3)}\Big(\sum_{k_1=\mu+\lambda}ik_1 e^{-itQ_1(k_1,\mu,\lambda)}\hat{v}_{\mu}\hat{v}_{\lambda}\Big)\hat{v}_{k_2}\hat{v}_{k_3}\nonumber\\
&=\sum_{k_{1}+k_{2}+k_{3}+k_4=k}\frac{ik(k_1+k_2+k_3)(k_1+k_2)
e^{-itQ_3(k,k_1,k_2,k_3,k_4)}}{Q_1(k,k_1+k_2+k_3,k_4)Q_2(k,k_1+k_2,k_3,k_4)}\hat{v}_{k_1}\hat{v}_{k_2}\hat{v}_{k_3}\hat{v}_{k_4},\label{99-10}
\end{align}
where $Q_3(k,k_1,k_2,k_3,k_4)=Q_2(k,k_1,k_2,k_3)+Q_1(k_1+k_2,k_1,k_2).$ In the last line, we rename the variables as $k_1=\mu$, $k_2=\lambda$, $k_3=k_2$, $k_4=k_3.$

Similarly, the second term corresponding to
$\partial_t\hat{v}_{k_2}$ can be obtained by the same method. Then the third term corresponding to $\partial_t\hat{v}_{k_3}$ is
\begin{align}
&\>\sum_{k_{1}+k_{2}+k_{3}=k}\frac{k(k_1+k_2)e^{-itQ_2(k,k_1,k_2,k_3)}}{Q_1(k,k_1+k_2,k_3)Q_2(k,k_1,k_2,k_3)}
\partial_t\hat{v}_{k_3}\hat{v}_{k_1}\hat{v}_{k_2}\nonumber\\
&=\sum_{k_1+k_{2}+k_{3}+k_4=k}\frac{ik(k_1+k_2)(k_3+k_4)e^{-itQ_3(k,k_1,k_2,k_3,k_4)}}{Q_1(k,k_1+k_2,k_3+k_4)Q_2(k,k_1,k_2,k_3+k_4)}
\hat{v}_{k_1}\hat{v}_{k_2}\hat{v}_{k_3}\hat{v}_{k_4},\label{99-11}
\end{align}
Hence, combining the resonant term decomposition \eqref{0.2} and the differentiation by parts procedure for the non-resonant term \eqref{99-10} and \eqref{99-11}, we obtain the complete structure of \eqref{99-9}
\begin{align*}
\widehat{B}_3(v,v,v)_k=-\partial_t\widehat{A}_3(v,v,v)_k+\widehat{R}^0_{3}(v,v,v)_{k}+2\widehat{A}^1_4(v,v,v,v)_k+\widehat{A}^2_4(v,v,v,v)_k,
\end{align*}
where the operators are defined as follows:
\begin{align}
\widehat{A}_3(u,v,w)_k&=\sum_{k_{1}+k_{2}+k_{3}=k}\frac{k (k_1+k_2) e^{-itQ_2(k,k_1,k_2,k_3)}}{Q_1(k,k_1+k_2,k_3)Q_2(k,k_1,k_2,k_3)}\hat{v}_{k_1}\hat{w}_{k_2}\hat{u}_{k_3}.\label{0.3}\\
\widehat{A}_4^1(u,v,w,z)&=\sum_{k_{1}+k_{2}+k_{3}+k_4=k}\frac{ik(k_1+k_2+k_3)(k_1+k_2)
e^{-itQ_3(k,k_1,k_2,k_3,k_4)}}{Q_1(k,k_1+k_2+k_3,k_4)Q_2(k,k_1+k_2,k_3,k_4)}\hat{u}_{k_1}\hat{v}_{k_2}\hat{w}_{k_3}\hat{z}_{k_4}.\nonumber\\
\widehat{A}_4^2(u,v,w,z)&=\sum_{k_1+k_{2}+k_{3}+k_4=k}\frac{ik(k_1+k_2)(k_3+k_4)e^{-itQ_3(k,k_1,k_2,k_3,k_4)}}{Q_1(k,k_1+k_2,k_3+k_4)Q_2(k,k_1,k_2,k_3+k_4)}
\hat{u}_{k_1}\hat{v}_{k_2}\hat{w}_{k_3}\hat{z}_{k_4}.\nonumber
\end{align}
If we put everything together and combine the terms into \eqref{1-7}, %i.e., substituting \eqref{1-8} into \eqref{1-7},
 we obtain
\begin{align}
\hat{v}_k(t)=\hat{v}_k(0)-\widehat{A}(v)_k(t)+\widehat{A}(v)_k(0)+2\int_0^t\widehat{B}(v)_k(\tau)d\tau,\label{1-16}
\end{align}
where
\begin{align*}
&\widehat{A}(v)_k=\widehat{A}_2(v,v)_k+\widehat{A}_3(v,v,v)_k,\\
&\widehat{B}(v)_k=\widehat{R}^0_{3}(v,v,v)_{k}+2\widehat{A}^1_4(v,v,v,v)_k+\widehat{A}^2_4(v,v,v,v)_k.
\end{align*}

Combining the expression of \eqref{1-5} and \eqref{1-16}, we will consider high and low frequencies separately  to close the contraction.
Fix $N$ large to be determined later. Define the operator $\Phi$ as follows:
\begin{align}
\Phi(v)_k(t)=
\left\{
\begin{array}{ll}
\hat{v}_k(0)-\widehat{A}(v)_k(t)+\widehat{A}(v)_k(0)+\displaystyle2\int_0^t\widehat{B}(v)_k(\tau)d\tau & \text{if } |k|>N,\\
\>
\hat{v}_k(0)+\displaystyle\int_0^t{ik}\displaystyle\sum_{k_1+k_2=k}e^{-3i\tau kk_1k_2}e^{2\tau\varepsilon k_1k_2}\hat{v}_{k_1}\hat{v}_{k_2}d\tau
& \text{if } |k|\leq N.
\end{array}\right.\label{1-13}
\end{align}

Using the relation $u(t)=e^{-t(\partial_x^3-\varepsilon\partial_x^2)}v(t)$, denote $\tilde{\Phi}(v)=\mathcal{F}^{-1}[\Phi(v)_k]$, we have
\begin{align*}
\tilde{\Phi}(u)(t)=e^{-t(\partial_x^3-\varepsilon\partial_x^2)}\tilde{\Phi}(v)(t).
\end{align*}
We note that $e^{-t\partial_x^3}$ is a linear isometry on $L^2$ for all $t\in R$, therefore
\begin{align*}
\|\tilde{\Phi}(u)(t)\|_{H^s}=\|e^{t\varepsilon\partial_x^2}\tilde{\Phi}(v)(t)\|_{H^s}.
\end{align*}
Before proving the theorem we aim to obtain, we need a lemma to assist in the proof.
At this point, we need some estimates.
\begin{lem}\label{0-9}
Let $s\geq 0$, we have the following inequalities hold:
\begin{align*}
&\|\mathbb{P}_{>N}e^{t\varepsilon\partial_x^2}A_2(u,v)\|_{H^{s}}\lesssim N^{-1}\|e^{t\varepsilon\partial_x^2}u\|_{H^s}\|e^{t\varepsilon\partial_x^2}v\|_{H^s}.\\
&\|\mathbb{P}_{>N}e^{t\varepsilon\partial_x^2}A_3(u,v,w)\|_{H^{s}}\lesssim N^{-{1}}\|e^{t\varepsilon\partial_x^2}u\|_{H^s}\|e^{t\varepsilon\partial_x^2}v\|_{H^s}\|e^{t\varepsilon\partial_x^2}w\|_{H^s}.\\
&\|\mathbb{P}_{>N}e^{t\varepsilon\partial_x^2}B^0_3(u,v,w)\|_{H^{s}}\lesssim \|e^{t\varepsilon\partial_x^2}u\|_{H^s}\|e^{t\varepsilon\partial_x^2}v\|_{H^s}\|e^{t\varepsilon\partial_x^2}w\|_{H^s}.\\
&\|\mathbb{P}_{>N}e^{t\varepsilon\partial_x^2}A^{i}_{4}(u,v,w,z)\|_{H^{s}}\lesssim\|e^{t\varepsilon\partial_x^2}u\|_{H^s}
\|e^{t\varepsilon\partial_x^2}v\|_{H^s}\|e^{t\varepsilon\partial_x^2}w\|_{H^s}\|e^{t\varepsilon\partial_x^2}z\|_{H^s},i=1,2.
\end{align*}
\end{lem}
\begin{proof}
For the formula $Q_1(k,k_1,k_2)=3kk_1k_2+i\varepsilon(k^2-k_1^2-k_2^2)$, we apply Lemma \ref{0-5} to obtain
 $|Q_1(k,k_1,k_2)|\gtrsim |k||k_1||k_2|$. The symmetry of $k_1$ and $k_2$ allows us to assume that $|k_1|\geq |k_2|$, which implies $|k_1|\gtrsim|k|$. Thus, we have $|k_1|^s\gtrsim|k|^s$ for $s\geq 0$, then
\begin{align*}
\|e^{t\varepsilon\partial_x^2}A_2(u,v)\|_{H^s}&=\||k|^s\mathcal{F}_k[e^{t\varepsilon\partial_x^2}A_2(u,v)]\|_{L^2}\\
&\lesssim \Big\|\sum_{k_{1}+k_{2}=k}
\frac{|k^{s+1}||e^{-3ikk_1k_2t}|}{|Q(k,k_1,k_2)|}
\mathcal{F}_{k_1}[e^{t\varepsilon \partial_x^2}u]\mathcal{F}_{k_2}[{e^{t\varepsilon\partial_x^2}}v]\Big\|_{L^2}\\
&\lesssim \Big\|\sum_{k_{1}+k_{2}=k}
|k|^s|k_1|^{-1}|k_2|^{-1}\mathcal{F}_{k_1}[e^{t\varepsilon \partial_x^2}u]\mathcal{F}_{k_2}[{e^{t\varepsilon \partial_x^2}}v]\Big\|_{L^2}\\
&\lesssim\Big\|\sum_{k_{1}+k_{2}=k}
\frac{|k_1|^s|k_2|^{s}}{|k_1||k_2|^{1+s}}\mathcal{F}_{k_1}[e^{t\varepsilon \partial_x^2}u]\mathcal{F}_{k_2}[{e^{t\varepsilon \partial_x^2}}v]\Big\|_{L^2}\\
&\lesssim \||k|^{-1}|k|^s\mathcal{F}_{k}[e^{t\varepsilon \partial_x^2}u]\|_{L^2}\Big\|\frac{|k|^s\mathcal{F}_{k}[{e^{t\varepsilon \partial_x^2}}v]}{|k|^{1+s}}\Big\|_{L^1}.
\end{align*}
Since $s\geq0$, we have $|k|^{-1-s}\lesssim |k|^{-1}$.
And in the last line, using Young's inequality and Cauchy-Schwarz inequality, we have
\begin{align*}
\|\mathbb{P}_{>N}e^{t\varepsilon\partial_x^2}A_2(u,v)\|_{H^s}
\lesssim N^{-1}\|e^{t\varepsilon \partial_x^2}u\|_{H^s}\|e^{t\varepsilon \partial_x^2}v\|_{H^s}.
\end{align*}

For $A_3(u,v,w)_k$ and $R_3^0(u,v,w)_k$ defined in \eqref{0.2} and \eqref{0.3}, by Lemma \ref{0-5}, we have $|Q_1|^{-1}\lesssim (|k||k_1||k_2|)^{-1}$ and $|Q_2|^{-1}\lesssim |\tilde{\phi}|^{-1}$, $\tilde{\phi}=3(k_1+k_2)(k_1+k_3)(k_2+k_3)$, then we obtain
\begin{align*}
&|\mathcal{F}_k[e^{t\varepsilon \partial_x^2}A_3(u,v,w)]|
\lesssim \sum_{k_3\neq 0,\Gamma(k)}\frac{\big|\mathcal{F}_{k_1}[e^{t\varepsilon \partial_x^2}u]\mathcal{F}_{k_2}[e^{t\varepsilon \partial_x^2}v]\mathcal{F}_{k_3}[e^{t\varepsilon \partial_x^2}w]\big|}{|k_3||(k_1+k_2)(k_1+k_3)(k_2+k_3)|}.\\\label{099}
&|\mathcal{F}_k[e^{t\varepsilon\partial_x^2}R_3^0(u,v,w)]|
\lesssim  \sum_{k_3\neq 0, \Gamma_0(k)}|k_3|^{-1}\big|\mathcal{F}_{k_1}[e^{t\varepsilon\partial_x^2}u]\mathcal{F}_{k_2}[e^{t\varepsilon\partial_x^2}v]
\mathcal{F}_{k_3}[e^{t\varepsilon\partial_x^2}w]\big|.
\end{align*}
Using Lemma \ref{0-6}, we can obtain
\begin{align*}
\|\mathbb{P}_{>N}e^{t\varepsilon\partial_x^2}R_3^0(u,v,w)\|_{H^s}
&\lesssim\|e^{t\varepsilon\partial_x^2}u\|_{L_t^{\infty}H_x^s}\|e^{t\varepsilon\partial_x^2}u\|_{L_t^{\infty}H_x^s}
\|e^{t\varepsilon\partial_x^2}u\|_{L_t^{\infty}H_x^s},\\
\|\mathbb{P}_{>N}e^{t\varepsilon \partial_x^2}A_3(u,v,w)\|_{L_t^{\infty}H_x^s}
&\lesssim N^{-1}\|e^{t\varepsilon \partial_x^2}u\|_{L_t^{\infty}H_x^s}\|e^{t\varepsilon \partial_x^2}v\|_{L_t^{\infty}H_x^s}\|e^{t\varepsilon \partial_x^2}w\|_{L_t^{\infty}H_x^s}.
\end{align*}

For $e^{t\varepsilon\partial_x^2}A^i_4(t)$, since $(k_1+k_2,k_3,k_4)\in \Gamma(k)$, we have $(k_1+k_2+k_3)(k_1+k_2+k_4)(k_3+k_4)\neq 0.$ Then we can obtain
\begin{align}
&|\mathcal{F}_k[e^{t\varepsilon\partial_x^2}A^1_4(u,v,w,z)]|= |e^{-t\varepsilon k^2}A^1_4(u,v,w,z)_k|\nonumber\\
&\lesssim \sum_{k_1+k_{2}+k_{3}+k_4=k}\frac{|k_1+k_2||\mathcal{F}_{k_1}[e^{t\varepsilon\partial_x^2}u]\mathcal{F}_{k_2}[e^{t\varepsilon\partial_x^2}v]
\mathcal{F}_{k_3}[e^{t\varepsilon\partial_x^2}w]\mathcal{F}_{k_4}[e^{t\varepsilon\partial_x^2}z]|}{|k_4||k_3+k_4||k_1+k_2+k_3||k_1+k_2+k_4|}.
\end{align}
Since
$|k_1+k_2|\lesssim |k_1+k_2+k_4|+|k_4|$,
\begin{align*}
|\mathcal{F}_k[e^{t\varepsilon\partial_x^2}A^1_4(u,v,w,z)]|%=\left|e^{-t\varepsilon k^2}A^1_4(u,v,w,z)_k\right|\\
&\lesssim
\sum_{k_1+k_{2}+k_{3}+k_4=k}\frac{\mathcal{F}_{k_1}[e^{t\varepsilon\partial_x^2}u]\mathcal{F}_{k_2}[e^{t\varepsilon\partial_x^2}v]
\mathcal{F}_{k_3}[e^{t\varepsilon\partial_x^2}w]\mathcal{F}_{k_4}[e^{t\varepsilon\partial_x^2}z]}{|k_4||k_3+k_4||k_1+k_2+k_3|}\\
&\quad+\sum_{k_1+k_{2}+k_{3}+k_4=k}\frac{\mathcal{F}_{k_1}[e^{t\varepsilon\partial_x^2}u]\mathcal{F}_{k_2}[e^{t\varepsilon\partial_x^2}v]
\mathcal{F}_{k_3}[e^{t\varepsilon\partial_x^2}w]\mathcal{F}_{k_4}[e^{t\varepsilon\partial_x^2}z]}{|k_3+k_4||k_1+k_2+k_3||k_1+k_2+k_4|},
\end{align*}
and
\begin{align*}
&|\mathcal{F}_k[e^{t\varepsilon\partial_x^2}A^2_4(u,v,w,z)]|
&\lesssim\sum_{k_1+k_{2}+k_{3}+k_4=k}\frac{\mathcal{F}_{k_1}[e^{t\varepsilon\partial_x^2}u]\mathcal{F}_{k_2}[e^{t\varepsilon\partial_x^2}v]
\mathcal{F}_{k_3}[e^{t\varepsilon\partial_x^2}w]\mathcal{F}_{k_4}[e^{t\varepsilon\partial_x^2}z]}{|k_3+k_4||k_1+k_2+k_3||k_1+k_2+k_4|}.
\end{align*}
Both terms of $A^1_4$ and $A^2_4$ are treated in the same way and satisfy the same bound by using Cauchy-Schwartz inequality. For $A^2_4$,
by duality, it suffices to estimate
\[
|\langle e^{t\varepsilon\partial_x^2}A^2_4(u,v,w,z), h\rangle| \lesssim \|u\|_{{H}^s} \|v\|_{{H}^s} \|z\|_{{H}^{s}}\|h\|_{H^{-s}}.
\]
where \( h \) is an arbitrary element in \( {H}^{-s} \).

Setting below \( \widehat{\tilde{u}}_k = \hat{u}_k |k|^s \), \( \widehat{\tilde{v}}_k = \hat{v}_k |k|^s \), \( \widehat{\tilde{w}}_k = \hat{w}_k |k|^s \), \( \widehat{\tilde{z}}_k = \hat{z}_k |k|^s \) and \( \widehat{\tilde{h}}_k = \hat{h}_k |k|^{-s} \) for each element $k$ in the sequence.
This yields that
\begin{align*}
&
\big|\langle e^{t\varepsilon\partial_x^2}A^2_4(u,v,w,z), h\rangle\big| \\
&\lesssim \sum_k \sum_{k_1 + k_2+k_3+k_4 = k} \frac{\mathcal{F}_{k_1}[e^{t\varepsilon\partial_x^2}u]\mathcal{F}_{k_2}[e^{t\varepsilon\partial_x^2}v]
\mathcal{F}_{k_3}[e^{t\varepsilon\partial_x^2}w]\mathcal{F}_{k_4}[e^{t\varepsilon\partial_x^2}z]\mathcal{F}_{k}[h]}
{|k_3+k_4||k_1+k_2+k_3||k_1+k_2+k_4|}\\
&\lesssim\sum_{k_1,k_2,k_3,k_4}
\frac{|k|^s\mathcal{F}_{k_1}[e^{t\varepsilon\partial_x^2}\tilde{u}]\mathcal{F}_{k_2}[e^{t\varepsilon\partial_x^2}\tilde{v}]
\mathcal{F}_{k_3}[e^{t\varepsilon\partial_x^2}\tilde{w}]\mathcal{F}_{k_4}[e^{t\varepsilon\partial_x^2}\tilde{z}]
\mathcal{F}_{{k_1+k_2+k_3+k_4}}[\tilde{h}]}{|k_1|^s|k_2|^s|k_3|^s|k_4|^s|k_3+k_4||k_1+k_2+k_3||k_1+k_2+k_4|}.
\end{align*}
Let $M=\max\{|k_1|,|k_2|,|k_3|,|k_4|\}$, since  $(|k_1||k_2||k_3||k_4|)^{-s}\lesssim M^{-s}$ and $|k|^s\lesssim M^s$ for $s\geq 0$, then we have
$\frac{|k|^s}{|k_1|^s|k_2|^s|k_3|^s|k_4|^s}\lesssim 1$.

 By applying the Cauchy-Schwarz  inequality, we can derive the following result:
\begin{align*}
\bigl| \bigl\langle e^{t\varepsilon\partial_x^2} A^2_4(u,v,w,z), h \bigr\rangle \bigr|
&\lesssim \biggl( \sum_{{k_1,k_2, k_3,k_4}}
    \Bigl| \mathcal{F}_{k_1}[e^{t\varepsilon\partial_x^2}\tilde{u}] \,
    \mathcal{F}_{k_2}[e^{t\varepsilon\partial_x^2}\tilde{v}] \,
    \mathcal{F}_{k_3}[e^{t\varepsilon\partial_x^2}\tilde{w}] \,
    \mathcal{F}_{k_4}[e^{t\varepsilon\partial_x^2}\tilde{z}] \Bigr|^2 \biggr)^{\!\!\frac{1}{2}} \\
&\quad \times \biggl( \sum_{{k_1,k_2,k_3,k_4}}
    \frac{ \bigl| \widehat{\tilde{h}}_{k_1+k_2+k_3+k_4} \bigr|^2 }
    { |k_3 + k_4|^2 \, |k_1 + k_2 + k_3|^2 \, |k_1 + k_2 + k_4|^2 } \biggr)^{\!\!\frac{1}{2}} \\
&\leq \bigl\| e^{t\varepsilon\partial_x^2} u \bigr\|_{H^s}
    \bigl\| e^{t\varepsilon\partial_x^2} v \bigr\|_{H^s}
    \bigl\| e^{t\varepsilon\partial_x^2} w \bigr\|_{H^s}
    \bigl\| e^{t\varepsilon\partial_x^2} z \bigr\|_{H^s}
    \| h \|_{H^{-s}}.
\end{align*}
Then we have
\begin{align*}
&\|e^{t\varepsilon\partial_x^2}A^i_4(u,v,w,z)\|_{H^s}\lesssim \|e^{t\varepsilon\partial_x^2}u\|_{L_t^{\infty}H_x^s}\|e^{t\varepsilon\partial_x^2}v\|_{L_t^{\infty}H_x^s}
\|e^{t\varepsilon\partial_x^2}w\|_{L_t^{\infty}H_x^s}\|e^{t\varepsilon\partial_x^2}z\|_{L_t^{\infty}H_x^s}, \> i={1,2}.
\end{align*}
This completes the proof of Lemma \ref{0-9}.
\end{proof}

Then for $|k|>N$, by the expression of \eqref{1-16} and \eqref{1-13}, we have
\begin{align}
\sup_{t \in [0,T]} \bigl\| e^{t\varepsilon\partial_x^2} \Phi(v)(t) \bigr\|_{H^s}
&\lesssim \| v(0) \|_{H^s}
    + \bigl\| e^{t\varepsilon\partial_x^2} A(v)(t) \bigr\|_{H^s}
    + \int_0^t \bigl\| e^{t\varepsilon\partial_x^2} B(v)(\tau) \bigr\|_{H^s}  \mathrm{d}\tau \nonumber \\
&\lesssim \| v(0) \|_{H^s}
    + N^{-1} \Bigl( \bigl\| e^{t\varepsilon\partial_x^2} v \bigr\|^2_{L_t^{\infty} H_x^s}
    + \bigl\| e^{t\varepsilon\partial_x^2} v \bigr\|^3_{L_t^{\infty} H_x^s} \Bigr) \nonumber \\
&\quad + t \Bigl( \bigl\| e^{t\varepsilon\partial_x^2} v \bigr\|^3_{L_t^{\infty} H_x^s}
    + \bigl\| e^{t\varepsilon\partial_x^2} v \bigr\|^4_{L_t^{\infty} H_x^s} \Bigr) \label{0.4}
\end{align}

For $|k|\leq N$, $|k|^s\lesssim |k_1|^s$ with $s\geq0$, we have
\begin{align*}
&\Big\|\sum_{k_1+k_2=k}ike^{-t\varepsilon k^2}e^{-3ikk_1k_2t+t\varepsilon(k^2-k_1^2-k_2^2)}v_{k_1}v_{k_2}\Big\|_{H^s}\\
&\lesssim \Big\|\sum_{k_1+k_2=k}|k|^{1+s}e^{-3ikk_1k_2t}\mathcal{F}_{k_1}[e^{t\varepsilon\partial_x^2}v]\mathcal{F}_{k_2}[e^{t\varepsilon\partial_x^2}v]
\Big\|_{L^2}\\
&\lesssim
\Big\|\sum_{k_1+k_2=k}|k||k_1|^s\mathcal{F}_{k_1}[e^{t\varepsilon\partial_x^2}v]\mathcal{F}_{k_2}[e^{t\varepsilon\partial_x^2}v]
\Big\|_{L^2}\\
&\lesssim
\|e^{t\varepsilon\partial_x^2}v\|_{H^s}\|e^{t\varepsilon\partial_x^2}v\|_{H^s}\|k\|_{L^2}\lesssim N^{{3}/{2}}\|e^{t\varepsilon\partial_x^2}v\|^2_{H^s}.
\end{align*}
Then for $|k|\leq N$, by \eqref{1-13}, we have
\begin{align}
\|e^{t\varepsilon\partial_x^2}\Phi(v)(t)\|_{H^s}\lesssim
&\|v(0)\|_{H^s}+tN^{{3}/{2}}\bigl\| e^{t\varepsilon\partial_x^2} v \bigr\|^2_{L_t^{\infty} H_x^s}.\label{0.5}
\end{align}
Combining \eqref{0.4} and \eqref{0.5}, we have the $H^s$ estimates for $e^{t\varepsilon\partial_x^2}\Phi(v)(t)$
\begin{align}\label{99-12}
\|e^{t\varepsilon\partial_x^2}\Phi(v)(t)\|_{L_t^{\infty}H_x^s}\lesssim \|v(0)\|_{H^s}+&(tN^{{3}/{2}}+N^{-1})
\Big(\bigl\|e^{t\varepsilon\partial_x^2}v\bigr\|^2_{L_t^{\infty}H_x^s}+\bigl\| e^{t\varepsilon\partial_x^2} v \bigr\|^4_{L_t^{\infty} H_x^s}\Big).
\end{align}

Define $X=
\{
u\in C([0,T];H^s):\|u\|_{L_{[0,T]}^{\infty}H^s}\leq 2Cr\},$ with $r=\|u(0)\|_{H^s}$, then we prove that $\Phi$ is a contraction on $X$.
From the expression of \eqref{99-12}, if $u\in X$, then
\begin{align*}
\|\Phi (u)\|_{L_{[0,T]}^{\infty}H^s}
&\lesssim \|u(0)\|_{H^s}+\big(N^{-1}+TN^{{3}/{2}}\big)\Big(\bigl\|e^{t\varepsilon\partial_x^2}v\bigr\|^2_{L_t^{\infty}H_x^s}+\bigl\| e^{t\varepsilon\partial_x^2} v \bigr\|^4_{L_t^{\infty} H_x^s}\Big).
\end{align*}
Choosing $N=T^{-{2}/{5}}$ such that $TN^{{3}/{2}}=N^{-1}$ in the inequality above, we have that
\begin{align*}
\|\Phi (u)\|_{L_{[0,T]}^{\infty}H^s}
&\leq C\|u(0)\|_{H^s}+CT^{{2}/{5}}(\|u\|^2_{L_{t}^{\infty}H_x^s}+\|u\|^4_{L_t^{\infty}H_x^s})\\
&\leq Cr+CT^{{2}/{5}}\left((2Cr)^2+(2Cr)^4\right)\leq  2Cr,
\end{align*}
if we fix $T$, provided $r$ satisfies $T^{{2}/{5}}(2C)^2r\leq 1/4$ and $T^{{2}/{5}}(2C)^4r^3\leq 1/4$.
And by the following estimate in Lemma \ref{eq1}, then given $u,v \in X$, we have
\begin{align*}
\|\Phi(u) - \Phi(v)\|_{L^\infty_T H^s} &\leq C T^{{2}/{5}}\Big( \|u\|_{L^\infty_T H^s} + \|v\|_{L^\infty_T H^s} + ( \|u\|_{L^\infty_T H^s} + \|v\|_{L^\infty_T H^s} )^3\Big) \|u - v\|_{L^\infty_T H^s} \\
&\leq C T^{{2}/{5}}\Big((2C)^2r+ (2C)^4r^3\Big) \|u - v\|_{L^\infty_T H^s} \\
&\leq \frac{1}{2} \|u - v\|_{L^\infty_T H^s}.
\end{align*}
Thus $\Phi$ is contraction mapping on $X$. This gives us a
unique solution in $X$ and continuous dependence on initial data for the equation $\Phi u = u$.
 This means that the existence time $T$ of the solution $u(t)$ is independent of $\varepsilon$.

Then we deduce that
\begin{align*}
\|u\|_{L_t^{\infty}H_x^s}\lesssim \|u(0)\|_{H^s}\lesssim \|\phi\|_{H^s},
\end{align*}
where the implicit constant is independent of $\varepsilon$.

\section{Uniform continuity and limit behavior}

In this section we complete the proof of Theorem \ref{00-1} and Theorem \ref{00-2}. We need to show the solution mapping $S^{\varepsilon}$ is continuous from $H^s\rightarrow C([0,T ];H^s)$ uniformly on $\varepsilon\in (0,1]$ and the inviscid limit in $H^s$.

\subsection{Uniform continuity}\mbox{}

We now prove the uniform continuity of the solution mapping \( S^\varepsilon \) at \( \phi \in H^s \) with respect to \( \varepsilon \). Specifically, for any \( \eta > 0 \), there exists \( \delta > 0 \) such that if \( \phi_1, \phi_2 \in H^s \) satisfy \( \|\phi_1 - \phi_2\|_{H^s} < \delta \), then for all \( \varepsilon \in (0, 1] \),
\[
\|S^\varepsilon(\phi_1) - S^\varepsilon(\phi_2)\|_{C([0,T];H^s)} < \eta.
\]

Define \( u_i = e^{-t(\partial_x^3 - \varepsilon \partial_x^2)} v_i \) for \( i = 1, 2 \). The following lemma provides key estimates for nonlinear operator differences, leveraging the relation \(\|e^{t\varepsilon \partial_x^2} v_i\|_{H^s} = \|u_i\|_{H^s}\) and Lemma \ref{0-9}.
This notation satisfies that $A_2(v,v)=A_2(v)$, $A_3(v,v,v)=A_3(v)$, $B^3_0(v,v,v)=B^3_0(v)$, $A_4^i(v,v,v,v)=A_4^i(v).$
\begin{lem}\label{eq1}
For $s\geq 0$, the following bounds on nonlinear operator differences hold:
\begin{align*}
&\|\mathbb{P}_{>N}e^{t\varepsilon\partial_x^2}(A_2(v_1)-A_2(v_2))\|_{H^s}\lesssim N^{-1}(\|u_1\|_{H^s}+\|u_2\|_{H^s})\|u_1-u_2\|_{H^s},\\
&\|\mathbb{P}_{>N}e^{t\varepsilon\partial_x^2}(A_3(v_1)-A_3(v_2))\|_{H^s}\lesssim N^{-1}(\|u_1\|_{H^s}+\|u_2\|_{H^s})^2\|u_1-u_2\|_{H^s},\\
&\|\mathbb{P}_{>N}e^{t\varepsilon\partial_x^2}(B_3^0(v_1)-B_3^0(v_2))\|_{H^s}\lesssim (\|u_1\|_{H^s}+\|u_2\|_{H^s})^2\|u_1-u_2\|_{H^s},\\
&\|\mathbb{P}_{>N}e^{t\varepsilon\partial_x^2}(A_4^i(v_1)-A_4^i(v_2))\|_{H^s}\lesssim (\|u_1\|_{H^s}+\|u_2\|_{H^s})^3\|u_1-u_2\|_{H^s}, i=1,2.
\end{align*}
\end{lem}
\begin{proof}
Using the bilinearity of \( A_2 \),  we have
\begin{align*}
\|A_2(v_1)-A_2(v_2)\|_{H^s}
&\lesssim\|A_2(v_1-v_2,v_1)\|_{H^s}+\|A_3(v_1-v_2,v_2)\|_{H^s}.
\end{align*}
By Lemma \ref{0-9} and the norm equivalence \(\|e^{t\varepsilon \partial_x^2} v_i\|_{H^s} = \|u_i\|_{H^s}\),
\begin{align*}
\|\mathbb{P}_{>N}e^{t\varepsilon\partial_x^2}(A_2(v_1)-A_2(v_2))\|_{H^s}
&\lesssim N^{-1}(\|e^{t\varepsilon\partial_x^2}v_1\|_{H^s}+\|e^{t\varepsilon\partial_x^2}v_2\|_{H^s})\|e^{t\varepsilon\partial_x^2}(v_1-v_2)\|_{H^s}\\
&= N^{-1}(\|u_1\|_{H^s}+\|u_2\|_{H^s})\|u_1-u_2\|_{H^s}.
\end{align*}
By the trilinearity of \( A_3 \), we have
\[
A_3(v_1) - A_3(v_2) = A_3(v_1 - v_2, v_1, v_1) + A_3(v_2, v_1 - v_2, v_1) + A_3(v_2, v_2, v_1 - v_2).
\]
Applying Lemma \ref{0-9} to each term, we obtain
\begin{align*}
\|\mathbb{P}_{>N}e^{t\varepsilon\partial_x^2}(A_3(v_1)-A_3(v_2))\|_{H^s}
&\lesssim N^{-1}\left(\|u_1\|^2_{H^s}+\|u_2\|^2_{H^s}+\|u_1\|_{H^s}\|u_2\|_{H^s}\right)\|u_1-u_2\|_{H^s}\\
&\lesssim N^{-1}\left(\|u_1\|_{H^s}+\|u_2\|_{H^s}\right)^2\|u_1-u_2\|_{H^s}.
\end{align*}
Similarly, we bound the difference $B_3^0(u_1)-B_3^0(u_2)$ by
\begin{align*}
\|\mathbb{P}_{>N}e^{t\varepsilon\partial_x^2}(B_3^0(v_1)-B_3^0(v_2))\|_{H^s}
\lesssim \left(\|u_1\|_{H^s}+\|u_2\|_{H^s}\right)^2\|u_1-u_2\|_{H^s}.
\end{align*}
And for \( A_4^i \) (\( i = 1, 2 \)), analogous decomposition into four terms gives:
\begin{align*}
&\|\mathbb{P}_{>N}e^{t\varepsilon\partial_x^2}(A_4^i(v_1)-A_4^i(v_2))\|_{H^s}\\
&\lesssim\|e^{t\varepsilon\partial_x^2}A_4^i(v_1-v_2,v_1,v_1,v_1)\|_{H_{|k|>N}^s}+\|e^{t\varepsilon\partial_x^2}A_4^i(v_1-v_2,v_2,v_2,v_2)\|_{H_{|k|>N}^s}\\
&\quad+\|e^{t\varepsilon\partial_x^2}A_4^i(v_1-v_2, v_1, v_1, v_2)\|_{H_{|k|>N}^s}+\|e^{t\varepsilon\partial_x^2}A_4^i(v_1-v_2,v_1,v_2, v_2)\|_{H_{|k|>N}^s}\\
&\lesssim \left(\|u_1\|^3_{H^s}+\|u_2\|^3_{H^s}+\|u_1\|^2_{H^s}\|u_2\|_{H^s}+\|u_1\|_{H^s}\|u_2\|^2_{H^s}\right)\|u_1-u_2\|_{H^s}\\
&\lesssim \left(\|u_1\|_{H^s}+\|u_2\|_{H^s}\right)^3\|u_1-u_2\|_{H^s}.
\end{align*}
This completes the proof of Lemma \ref{eq1}.
\end{proof}
Similarly, according to Lemma \ref{eq1}, for \eqref{1-13}, we obtain the following estimates for the high-frequency components ($|k| > N$):
\begin{align}\label{99-13}
\|\mathbb{P}_{\leq N}e^{t\varepsilon\partial_x^2}(v_1-v_2)\|_{H^s}\lesssim \|\phi_1 - \phi_2\|_{H^s}+tN^{3/2}(\|u_1\|_{H^s}+\|u_2\|_{H^s})\|u_1-u_2\|_{H^s}.
\end{align}

Substituting these estimates into the evolution equation for \( u_1 - u_2 \) (cf. \eqref{1-13} and \eqref{99-13}), and recalling that \(\| u_1 - u_2\|_{H^s} = \|e^{t\varepsilon\partial_x^2}(v_1 - v_2)\|_{H^s}\), we obtain
\begin{align*}
\|u_1 - u_2\|_{H^s} \lesssim &\|\phi_1 - \phi_2\|_{H^s} \\
&+ \big( N^{-1} + TN^{3/2} \big) \left( \|u_1\|_{H^s} + \|u_2\|_{H^s} + (\|u_1\|_{H^s} + \|u_2\|_{H^s})^3 \right) \|u_1 - u_2\|_{H^s}.
\end{align*}
Choose \( N \) large enough such that $N^{-1}=TN^{3/2}$. For \( t \in [0,T] \), this implies that
\begin{align}\label{11-2}
\|u_1 - u_2\|_{H^s} \leq C(T, \|\phi_1\|_{H^s}, \|\phi_2\|_{H^s}) \|\phi_1 - \phi_2\|_{H^s},
\end{align}
with the constant independent of \( \varepsilon \).

\subsection{Uniform global well-posedness}\mbox{}

We will extend the uniform local solution obtained in the last section to a global solution. The standard way to use conservation law. Let $u$ be a smooth solution of the KdV-Burgers equation \eqref{00} with initial data $\phi$, multiply $u$ and integrate, then we have
\begin{align*}
\frac{1}{2}\|u(t)\|^2_{L^2}+\varepsilon\int_0^t\big\|\partial_x u(\tau)\big\|^2_{L^2}d\tau=\frac{1}{2}\|\phi\|^2_{L^2}.\label{2-2}
\end{align*}
Thus if $\phi\in L^2(\mathbb{T})$, we obtain
\begin{align*}
\sup_{t\in\mathbb{R}}\|u(t)\|_{L^2}\lesssim \|\phi\|_{L^2}.
\end{align*}
The above inequality holds for $L^2$-strong solution, then we get that \eqref{00} is uniformly globally well-posed.

\subsection{Limit behavior}\mbox{}

Finally, we prove the inviscid limit behavior in $H^s$ for $s\geq 0$. It is well known that KdV equation \eqref{1.4} is completely integrable and has infinite conservation laws, and as a corollary one obtains that let $v$ be a smooth solution to \eqref{1.4} with initial data $v_0$, $s\geq 0$, we have
\begin{align*}
\sup_{t\in\mathbb{R}}\|v(t)\|_{H^s}\lesssim \|v_0\|_{H^{s}}.
\end{align*}
There are less symmetries for \eqref{00}. We can still expect that the $H^k$
norm of the solution remains bounded for a finite time $T > 0$, since the dissipative term behaves well for $t > 0$. Now we prove for $k = 1$ which will suffice for our purpose.
Assume $u$ is a smooth solution to \eqref{00}. Let $H[u] = \int_{\mathbb{T}} (u_x)^2 + \frac{2}{3}u^3 +u^2 dx$, then by Eq. \eqref{00} and
partial integration
\begin{align*}
\frac{d}{dt} H[u] &= \int_{\mathbb{T}} \left( 2u_x \partial_x (u_t) + u^2 u_t + 2u u_t \right) dx \\
&= \int_{\mathbb{T}} 2u_x \partial_x\left( -u_{xxx} + \varepsilon \partial_x^{2} u  + 2uu_x \right) dx + \int_{\mathbb{T}} u^2 \left( -u_{xxx} + \varepsilon  \partial_x^{2} u + 2uu_x \right) dx \\
&\quad+ \int_{\mathbb{T}} 2u\left(  -u_{xxx} + \varepsilon \partial_x^{2} u  +2uu_x \right) dx \\
&= \int_{\mathbb{T}} \left( -2\varepsilon (\partial_x^2 u)^2 + \varepsilon u^2 \partial_x^2 u - 2\varepsilon (\partial_x u)^2 \right) dx \\
&\leq -\varepsilon \int_{\mathbb{T}} (\partial_x^2 u)^2dx + \varepsilon\int_{\mathbb{T}} u^2  \partial_x^2 u   dx,
\end{align*}

Thus we have
\[
\frac{d}{dt} H[u] + \frac{\varepsilon}{2} \| \partial_x^2 u \|_{L^2}^2 \lesssim \| u \|_{L^4}^4.
\]
Using Gagliardo-Nirenberg inequality
\[
\| u \|_{L^3}^3 \lesssim \| u \|_{L^2}^{5/2} \| u_x \|_{L^2}^{1/2}, \quad \| u \|_{L^4}^4 \lesssim \| u \|_{L^2}^3 \| u_x \|_{L^2},
\]
and Cauchy-Schwarz inequality, we get
\begin{align}\label{eq3}
\sup_{t \in [0,T]} \| u(t) \|_{H^1} + \varepsilon^{1/2} \left( \int_0^T \| \partial_x^2 u(\tau) \|_{L^2}^2  d\tau \right)^{1/2} \leq C(T, \|\phi\|_{H^1}), \quad \forall T > 0.
\end{align}

Let $S^{\varepsilon}_{}(\varphi)$ and $S_{}(\varphi)$ denote the nonlinear solution mappings of the Cauchy problems \eqref{00} and \eqref{1.4} that associate with any initial data $\varphi$. For convenience, we only give the proof of the case $s = 0$, since the proof of the case $s > 0$ is similar.
It suffices to prove
$$\lim_{\varepsilon\rightarrow 0} \|S^{\varepsilon}(\varphi)-S(\varphi)\|_{C([0,T];{L^2})}=0.$$
Assume $u$ is a solution to \eqref{00} obtained in the last section and $v$ is a solution to \eqref{1.4} with initial data $\varphi_1, \varphi_2 \in L^2$, $u=S^{\varepsilon}(\varphi_1)$, $v=S(\varphi_2)$. Let $w=u-v$, $w_0=\varphi_1-\varphi_2$, then $w$ satisfies
\begin{equation}\label{8-3}
\left\{
\begin{aligned}
&w_t + \partial_x^3w - \varepsilon \partial_x^{2}u-\partial_x\left(w\big(u + v\big)\right) = 0, \quad t \in [0,T], \, x \in \mathbb{T},\\
&w(0) = w_0.
\end{aligned}
\right.
\end{equation}
Then the term $-\varepsilon\partial_x^2 u$ can be viewed as the perturbation to the difference equation of the KdV equation.
Indeed, we perform an inner product with $w$ in \eqref{8-3}, then we have
\begin{align*}
\langle w_t,w\rangle+\langle\partial_x^3w,w\rangle-\langle\varepsilon\partial_x^2u,w\rangle
=\langle\partial_x((u + v)w),w\rangle.
\end{align*}
We obtain
\begin{align}
\label{99-14}
\frac{1}{2}\frac{d}{dt}\|w(t)\|^2_{L^2}
&=\varepsilon\int_{\mathbb{T}}\partial_x^2u\cdot wdx+\frac{1}{2}\int_{\mathbb{T}}\partial_x(u+v)\cdot w^2dx.
\end{align}
The contribution of the first term of the above right-hand side can be estimated by the Cauchy-Schwarz inequality and Young's inequality, then we get
$$\varepsilon\int_{\mathbb{T}}\partial_x^2u\cdot wdx\lesssim\varepsilon\|\partial_x^2u(t)\|_{L^2}\|w(t)\|_{L^2}\leq \frac{\varepsilon^2}{2}\|\partial_x^2u(t)\|^2_{L^2}+\frac{1}{2}\|w(t)\|^2_{L^2}.$$
For $u, v \in L^2(\mathbb{T})$, the second term is estimated via integration by parts, Fourier techniques, and variable transformations.
The details are provided in the Appendix. Then we get
\begin{align*}
\int_0^t\int_{\mathbb{T}}\partial_x(u+v)w^2dxd\tau
\lesssim \varepsilon^2\int_0^T \|\partial_xu\|^2_{L^2}d\tau+C(T, \|v\|_{L^2},\|u\|_{L^2})\|w\|^2_{L_t^{\infty}L_x^2}.
\end{align*}
Integrating the differential inequality \eqref{99-14} on $[0, T]$, we obtain
\begin{align*}
\sup_{t \in [0,T]} \|w(t)\|_{L^2}^2
&\lesssim \|w_0\|_{L^2}^2 + \varepsilon^2 \int_0^{T} \|u(\tau)\|_{H^2}^2  d\tau + C(T, \|v\|_{L^2},\|u\|_{L^2}) \|w\|_{L_t^{\infty}L_x^2}^2.
\end{align*}
From Theorem \ref{00-1}, we have $C(T, \|v\|_{L^2},\|u\|_{L^2}) < 1$ when $T\|u\|_{L^2} < 1/4$ and $T\|v\|_{L^2} < 1/4$, for $T<1$. Combining this with the regularity estimate \eqref{eq3} yields
\begin{align*}
\|w\|_{L_t^{\infty}L_x^2}
&\lesssim \|w_0\|_{L^2} + \varepsilon \|u\|_{L^2_t H^2_x} \\
&\lesssim \|w_0\|_{L^2} + \varepsilon^{1/2} C(T, \|\varphi_1\|_{H^1}, \|\varphi_2\|_{L^2}).
\end{align*}

For general $\varphi_1,\varphi_2\in L^2$, using the scaling, then we immediately get that for any $T>0$ such that
the difference estimate between solutions $u=S^{\varepsilon}(\varphi_1)$, $v=S(\varphi_2)$ satisfies
 \begin{align}\label{5-1}
 \|u-v\|_{C([0,T];L^2)}\lesssim \|\varphi_1-\varphi_2\|_{L^2}+\varepsilon^{1/2} C(T,\|\varphi_1\|_{H^1},\|\varphi_2\|_{L^2} ).
 \end{align}

Indeed, for any \( \eta > 0 \), there is a \( K > 0 \) such that
$
\|\varphi- \mathbb{P}_{\leqslant K} \varphi  \|_{L^2} \leq \eta/4.
$
Then it follows from the uniform continuity \eqref{11-2}
\begin{align*}
&\|S^{\varepsilon}(\varphi) - S^{\varepsilon}(\mathbb{P}_{\leq K}\varphi)\|_{C([0,T];L^2)}\leq \eta/4,\\
&\|S(\varphi) - S(\mathbb{P}_{\leq K}\varphi)\|_{C([0,T];L^2)}\leq \eta/4.
\end{align*}
Fixing this $K$ large enough, by taking $\varepsilon>0$ sufficiently small, we get from \eqref{5-1} that
\begin{align*}
\|S^{\varepsilon}(\mathbb{P}_{\leq K}\varphi) - S(\mathbb{P}_{\leq K}\varphi)\|_{C([0,T];L^2)}\lesssim \varepsilon^{1/2} C(T,K,\|\varphi\|_{L^2} )\leq \eta/4.
\end{align*}
Combining the above estimates, thus we obtain
\begin{align*}
\|S^{\varepsilon}(\varphi) - S(\varphi)\|_{C([0,T];L^2)}
&\leq \|S^{\varepsilon}(\varphi) - S^{\varepsilon}(\mathbb{P}_{\leq K}\varphi)\|_{C([0,T];L^2)} \\
&\quad + \|S^{\varepsilon}(\mathbb{P}_{\leq K}\varphi) - S(\mathbb{P}_{\leq K}\varphi)\|_{C([0,T];L^2)}\\
&\quad + \|S(\mathbb{P}_{\leq K}\varphi) - S(\varphi)\|_{C([0,T];L^2)}\leq \eta.
\end{align*}
The proof of Theorem \ref{00-2} is completed.

\section{Uniform local well-posedness for the mKdV-B equation}

Without loss of generality, we can assume that the integral of the initial value $u_0$ is zero,
i.e. $\mathbb{P}_0u_0=0$
and $\mathbb{P}u_0 = u_0$,
where $\mathbb{P}_0$ and $\mathbb{P}$ are the zero-mode and nonzero-mode projection
operators defined in \eqref{22-1}.

 Since $\mathbb{P}_0u_0=0$, the conservation law
$\int_{\mathbb{T}}u(t,x)dx=\int_{\mathbb{T}}u_0(x) dx$
of mKdV-Burgers equation implies that
$\mathbb{P}_0u(t,\cdot)=0$ for all $t\geq 0$.
By introducing the twisted variable $v(t,x)$
$$v(t,x): = e^{t(\partial_x^3-\varepsilon\partial_x^2)}u(t,x), \quad t\geq 0, x\in \mathbb{T},$$
the mKdV-Burgers equation in \eqref{8-1} can be written as
\begin{equation*}
\partial_t v(t,x)=e^{t(\partial_x^3-\varepsilon\partial_x^2)}\partial_x\left(e^{-t(\partial_x^3-\varepsilon\partial_x^2)}v(t,x)\right)^3,\quad t\geq 0, x\in\mathbb{T}.
\end{equation*}

By applying the Fourier transform and considering whether $(k_1+k_2)(k_1+k_3)(k_2+k_3)$ is zero, we can decompose the above equality into two parts:
\begin{align}
\partial_t\hat{v}_k&=\sum_{k_1+k_2+k_3=k}ike^{-it Q_2(k,k_1,k_2,k_3)}\hat{v}_{k_1}\hat{v}_{k_2}\hat{v}_{k_3}\label{6-1}\\
&=\sum_{(k_1,k_2,k_3)\in\Gamma(k)}
ike^{-it Q_2(k,k_1,k_2,k_3)}\hat{v}_{k_1}\hat{v}_{k_2}\hat{v}_{k_3}-
ike^{t\varepsilon (k^2-k_1^2-k_2^2-k_3^2)}|\hat{v}_{k}|^2\hat{v}_{k}\nonumber\\
&=:\widehat{\mathcal{N}}(v)_k+\widehat{\mathcal{R}}(v)_k,\nonumber
\end{align}
where $Q_2(k,k_1,k_2,k_3)=3(k_1+k_2)(k_1+k_3)(k_2+k_3)+i\varepsilon(k^2-k_1^2-k_2^2-k_3^2)$.

By integrating both sides from 0 to $t$,
\begin{align*}
\hat{v}_k(t)-\hat{v}_k(0)&=\int_0^t\widehat{\mathcal{R}}(v)_k(\tau)d\tau+\int_0^t\widehat{\mathcal{N}}(v)_k(\tau)d\tau.
\end{align*}
For $\widehat{\mathcal{N}}_k$, we can perform differentiation by parts since
$e^{-itQ_2(k,k_1,k_2,k_3)}=\partial_t\left(\frac{ie^{-itQ_2(k,k_1,k_2,k_3)}}{Q_2(k,k_1,k_2,k_3)}\right)$. Following an argument similar to that in Section 3,  then we obtain
\begin{equation}
\hat{v}_k(t)=\hat{v}_k(0)-\widehat{\mathcal{N}_1}(v)_k(t)+\widehat{\mathcal{N}_1}(v)_k(0)+\int_0^t\widehat{\mathcal{R}}(v)_k(\tau)d\tau
+\int_0^t\widehat{\mathcal{N}_2}(v)_k(\tau)d\tau,\label{0000}
\end{equation}
where
\begin{align}
&\widehat{\mathcal{N}_1}(v)_k=\frac{ke^{-itQ_2(k,k_1,k_2,k_3)}}{Q_2(k,k_1,k_2,k_3)}\hat{v}_{k_1}\hat{v}_{k_2}\hat{v}_{k_3},\label{9.1}\\
&\widehat{\mathcal{N}_2}(v)_k=\sum_{\substack{k=k_1+k_2+k_3\\ k_1=j_1+j_2+j_3}}\frac{ikk_1}{Q_2(k,k_1,k_2,k_3)} e^{-isQ_3(k,k_1,k_2,k_3,j_1,j_2,j_3)}\hat{v}_{j_1}\hat{v}_{j_2}\hat{v}_{j_3}\hat{v}_{k_2}\hat{v}_{k_3},\label{9.2}
\end{align}
where $Q_3(k,k_1,k_2,k_3,j_1,j_2,j_3)=Q_2(k,k_1,k_2,k_3)+Q_2(k_1,j_1,j_2,j_3).$

Combining \eqref{6-1} and \eqref{0000}, we will consider high and low frequencies separately to close the contraction.
Fix $N$ large to be determined later. Then we obtain
\begin{align*}
\hat{v}_k(t)=
\left\{
\begin{array}{ll}
\hat{v}_k(0)-\widehat{\mathcal{N}_1}(v)_k(t)+\displaystyle \widehat{\mathcal{N}_1}(v)_k(0)+\displaystyle\int_0^t\Big(\widehat{\mathcal{R}}(v)_{k}+\widehat{\mathcal{N}_2}(v)_k\Big)(\tau)d\tau & \text{if } |k|>N,\\
\>
\hat{v}_k(0)+\displaystyle\int_0^t\displaystyle\sum_{k_1+k_2+k_3=k}ike^{-i\tau Q_2(k,k_1,k_2,k_3)}\hat{v}_{k_1}\hat{v}_{k_2}\hat{v}_{k_3}d\tau
& \text{if } |k|\leq N.
\end{array}\right.
\end{align*}

\begin{lem}\label{8-4}
 Let $s\geq {1}/{2}$,  \( \tilde{\phi} \) be defined in \eqref{0-12}, and \( k = k_1 + k_2 + k_3 \), the multilinear  function $C_4(v)=C_4(v_1,v_2,v_3,v_4,v_5)$ satisfies the following inequality:
\begin{align*}
\left|\mathcal{F}_k[C_4(v)]\right|\lesssim \sum_{\substack{ k_1+k_2+k_3=k\\j_1+j_2+j_3=k_1, \>\Gamma(k)}} |k||k_1|  |\tilde{\phi}|^{-1}|\hat{v}_{j_1}\hat{v}_{j_2}\hat{v}_{j_3}\hat{v}_{k_2}\hat{v}_{k_3}|.
\end{align*}
Then the following estimate holds:
\[
\| C_4(v) \|_{H^s} \lesssim \prod_{i=1}^5\|v_i\|_{L^{\infty}_tH^s_x}.
\]
\end{lem}
\begin{proof}
Decompose the inequality into two parts since $\Gamma(k)=\Gamma_1(k)\cup\Gamma_2(k)$, then
\begin{align*}
\left|\mathcal{F}_k[C_4(v)]\right|\lesssim\left|\mathcal{F}_k[C_{41}(v)]\right|+\left|\mathcal{F}_k[C_{42}(v)]\right|.
\end{align*}
For $(k_1,k_2,k_3)\in \Gamma_1(k)$, we have $|k_1|\sim|k_2|\sim|k_3|\sim|k|$ and \( |\tilde{\phi}| \gtrsim |k_m|\Lambda \), where
$\Lambda$ defined in \eqref{99-7}. Assume $|k| \lesssim |j_1|$.
For $s\geq 1/2$, $|k|\lesssim|k|^{2s}$ holds. Applying Cauchy-Schwarz inequality
\begin{align*}
\|\mathcal{F}_k[C_{41}]\|^2_{H^s} &\lesssim \sum_{k}|k|^{2s}\Big(\sum_{k=k_1+k_2+k_3}|k|\Lambda^{-1}\Big(\sum_{k_1=j_1+j_2+j_3}|v_{j_1}v_{j_2}v_{j_3}|\Big)|v_{k_2}v_{k_3}|
\Big)^2\\
&\lesssim\sum_{k}|k|^{2s}\sum_{k_2,k_3}|k||v_{k_2}|^2|v_{k_3}|^2\sum_{k_2,k_3}\Big(\sum_{k-k_2-k_3=j_1+j_2+j_3}\Lambda^{-1}|v_{j_1}v_{j_2}v_{j_3}|\Big)^2
\\
&\lesssim \sum_{k}|k|^{2s}\sum_{k_2,k_3}|k|^{4s}|v_{k_2}|^{2}|v_{k_3}|^2
\sum_{k_1=j_1+j_2+j_3}|j_2||j_3|\sum_{j_2,j_3}|v_{j_1}|^2|v_{j_2}|^2|v_{j_3}|^2
\\
&\lesssim \sum_{k}|k|^{2s}\sum_{k_1,k_2,k_3}
|k_2|^{2s}|{k_3}|^{2s}|v_{k_2}|^2|v_{k_3}|^2\sum_{j_1,j_2,j_3} |j_2| |j_3| |\hat{v}_{j_1}|^2 |\hat{v}_{j_2}|^2 |\hat{v}_{j_3}|^2 \\
&\lesssim\sum_{j_1,j_2,j_3,k_2,k_3}|j_1|^{2s}|j_2|^{2s}|j_3|^{2s}|v_{j_1}|^2|v_{j_2}|^2|v_{j_3}|^2|k_2|^{2s}|v_{k_2}|^{2}|k_3|^{2s}|v_{k_3}|^2\\
&\lesssim\|v_1\|^2_{H^s}\|v_2\|^2_{H^s}\|v_3\|^2_{H^s}\|v_4\|^2_{H^s}\|v_5\|^2_{H^s}.
\end{align*}

 \textbf{Case $s > 1/2$ :} For $(k_1,k_2,k_3)\in\Gamma_2(k)$, we have $|\tilde{\phi}|\gtrsim|k_m|^2$.  Using Kato-Ponce inequality in Lemma \ref{7-1}, we obtain
\begin{align*}
\|\mathcal{F}_{k} [C_{42}(v)]\|_{H^s}&\lesssim \Big\|\sum_{\Gamma_2(k)}|k||k_1||k_m|^{-2}\hat{v}_{j_1}\hat{v}_{j_2}\hat{v}_{j_3}\hat{v}_{k_2}\hat{v}_{k_3}\Big\|_{H^s}\\
&\lesssim \Big\|\sum_{\Gamma_2(k)}\hat{v}_{k_1}\hat{v}_{k_2}\hat{v}_{k_3}\hat{v}_{k_4}\hat{v}_{k_5}\Big\|_{H^s}\lesssim\big\|\prod_{i=1}^5v_i\big\|_{H^s}\\
&\lesssim \|v_1\|_{H^s}\|v_2\|_{H^s}\|v_3\|_{H^s}\|v_4\|_{H^s}\|v_5\|_{H^s}.
\end{align*}
\textbf{Case $s = 1/2$ :} For $(k_1,k_2,k_3)\in\Gamma_2(k)$, we consider $|\tilde{\phi}|\gtrsim|k_m|^2\lambda$, where $\lambda=\lambda_k=:\min\{|k_1+k_2|,|k_2+k_3|,|k_1+k_3|\}$.
By duality, it suffices to estimate
\begin{align}
|\langle \mathcal{F}_{k}[ C_{42}(v)], h\rangle| \lesssim \|v_1\|_{{H}^{1/2}} \|v_2\|_{{H}^{1/2}} \|v_3\|_{{H}^{1/2}} \|v_4\|_{{H}^{1/2}} \|v_5\|_{{H}^{1/2}}\|h\|_{H^{-{1/2}}},
\end{align}
where \( h \) is an arbitrary element in \( {H}^{-1/2} \).

We further decompose $\Gamma_2(k)$ into two parts by Lemma \ref{5}, i.e., $\Gamma_2(k) = \Gamma_{21}(k) \cup \Gamma_{22}(k)$.

For $(k_1,k_2,k_3)\in \Gamma_{21}(k)$, we have $|k_m|^2\lesssim |\tilde{\phi}|\ll|k_m|^{\frac{15}{7}}$.
Without loss of generality, we may assume
$\lambda=|k_2+k_3|\leq\min\{|k_1+k_3|,|k_1+k_2|\}$, and $|k_2+k_3|\ll||k_m|^{5/7}$, then we have $|k_2+k_3|\gtrsim |k_m|^{\frac{4}{7}}$.
Then we have
\begin{align*}
|\langle \mathcal{F}_{k}[ C_{42}(v)], h\rangle| &\lesssim \sum_k\sum_{\Gamma_{21}(k)} \frac{|k||k_1|}{|k_m|^2\lambda}|\hat{v}_{j_1}||\hat{v}_{j_2}||\hat{v}_{j_3}||\hat{v}_{k_2}||\hat{v}_{k_3}||\hat{h}_k|\\
&\lesssim \sum_k\sum_{j_1+j_2+j_3+k_2+k_3=k} \lambda^{-1}|\hat{v}_{j_1}||\hat{v}_{j_2}||\hat{v}_{j_3}||\hat{v}_{k_2}||\hat{v}_{k_3}||\hat{h}_k|.
\end{align*}
 Setting below \( \widehat{{\tilde{v}}}_{j_1}= |j_1|^{1/2}\hat{{\tilde{v}}}_{j_1}  \), \( \widehat{\tilde{v}}_{j_2} = |j_2|^{-\delta} \hat{v}_{j_2} \), \( \widehat{\tilde{v}}_{j_3} = |j_3|^{-\delta}\hat{v}_{j_3}  \), \( \widehat{\tilde{v}}_{k_2} = |k_2|^{-\delta}\hat{v}_{k_2}  \), \( \widehat{\tilde{v}}_{k_3} =  |k|^{-\delta}\hat{v}_{k_3} \) and \( \widehat{\tilde{h}}_k = |k|^{-{1/2}} \hat{h}_k  \).
 Without loss of generality, assume
\begin{align*}
|j_1| &= \max(|j_1|, |j_2|, |j_3|) > |k_1|,\\
|k_1| &=|k_m| = \max(|k_1|, |k_2|, |k_3|, |k|).
\end{align*}
For sufficiently small $\delta>0,$ this yields that
\begin{align}
|\langle \mathcal{F}_{k}[ C_{42}(v)], h\rangle|
 &\lesssim \sum_k\sum_{j_1,j_2,j_3,k_2,k_3}\frac{|k|^{{1}/{2}}|j_2|^{\delta}|j_3|^{\delta}|k_2|^{\delta}|k_3|^{\delta}}{|j_1|^{{1}/{2}}|k_2+k_3|}
 |\widehat{{\tilde{v}}}_{j_1}||\widehat{{\tilde{v}}}_{j_2}||\widehat{{\tilde{v}}}_{j_3}||\widehat{{\tilde{v}}}_{k_2}|
 |\widehat{{\tilde{v}}}_{k_3}||\widehat{\tilde{h}}_k|\nonumber\\
&\lesssim  \sum_k\sum_{j_1,j_2,j_3,k_2,k_3}\frac{|k|^{{1}/{2}}|k_m|^{2\delta}}{|j_1|^{{1}/{2}-2\delta}|k_2+k_3|}
|\widehat{{\tilde{v}}}_{j_1}||\widehat{{\tilde{v}}}_{j_2}||\widehat{{\tilde{v}}}_{j_3}||\widehat{{\tilde{v}}}_{k_2}|
|\widehat{{\tilde{v}}}_{k_3}||\widehat{\tilde{h}}_k|\nonumber\\
&=
\sum_k\sum_{{\substack{k=k_1+k_2+k_3\\ k_1=j_1+j_2+j_3}}}\frac{|k_m|^{4\delta}}{|k_2+k_3|}
|\widehat{{\tilde{v}}}_{j_1}||\widehat{{\tilde{v}}}_{j_2}||\widehat{{\tilde{v}}}_{j_3}||\widehat{{\tilde{v}}}_{k_2}|
|\widehat{{\tilde{v}}}_{k_3}||\widehat{\tilde{h}}_k|\nonumber\\
&\lesssim
 \|{\tilde{v}}_{1}\|_{L^2}\|\tilde{h}\|_{L^2}\|\tilde{v}_2\|_{L^{\infty}}\|\tilde{v}_3\|_{L^{\infty}}\|\tilde{v}_4\|_{L^{\infty}}
 \|\tilde{v}_5\|_{L^{\infty}}.\label{eq6}
\end{align}
To complete the estimate, we analyze the multiplier coefficient:
\begin{align*}
M = \frac{|k_m|^{4\delta}}{|k_2 + k_3|} \lesssim |k_m|^{4\delta - 4/7} \leq 1,
\end{align*}
where the inequality holds because $|k_2 + k_3| \gtrsim |k_m|^{4/7}$ and by choosing $\delta < 1/7$.

Additionally, the modified Fourier coefficients satisfy the Sobolev embeddings:
\begin{align*}
\| \tilde{v}_1 \|_{L^2} &= \| |j_1|^{1/2} \hat{v}_{j_1} \|_{L^2} \lesssim \| v_1 \|_{H^{1/2}}, \\
\| \tilde{h} \|_{L^2} &= \| |k|^{-1/2} \hat{h}_k \|_{L^2} \lesssim \| h \|_{H^{-1/2}}, \\
\| \tilde{v}_i \|_{L^\infty} &\lesssim \| |\cdot|^{-\delta} \hat{v}_{i,\cdot} \|_{L^1} \lesssim \| v_i \|_{H^{1/2}} \quad (i=2,3,4,5).
\end{align*}

Combining these estimates with \eqref{eq6} yields that
\begin{align*}
|\langle \mathcal{F}[C_{42}(v)], h\rangle|
\lesssim \| v_1 \|_{H^{1/2}} \| h \|_{H^{-1/2}} \prod_{i=2}^5 \| v_i \|_{H^{1/2}} = \prod_{i=1}^5 \| v_i \|_{H^{1/2}} \| h \|_{H^{-1/2}}.
\end{align*}

For $(k_1,k_2,k_3)\in \Gamma_{22}(k)$, we have $|\tilde{\phi}|\gtrsim|k_m|^{{15}/{7}}$.
Then we have
\begin{align*}
|\langle \mathcal{F}_{k}[ C_{42}(v)], h\rangle| &\lesssim \sum_k\sum_{\Gamma_{22}(k)} \frac{|k||k_1|}{|k_m|^{{15}/{7}}}|\hat{v}_{j_1}||\hat{v}_{j_2}||\hat{v}_{j_3}||\hat{v}_{k_2}||\hat{v}_{k_3}||\hat{h}_k|\\
&\lesssim
\sum_{j_1,j_2,j_3,k_2,k_3} |k_m|^{-1/7}|\hat{v}_{j_1}||\hat{v}_{j_2}||\hat{v}_{j_3}||\hat{v}_{k_2}||\hat{v}_{k_3}||\hat{h}_k|\\
&\lesssim \sum_{j_1,j_2,j_3,k_2,k_3}\frac{|k|^{{1}/{2}}|j_2|^{\delta}|j_3|^{\delta}|k_2|^{\delta}|k_3|^{\delta}}{|j_1|^{{1}/{2}}|k_m|^{-1/7}}
 |\widehat{{\tilde{v}}}_{j_1}||\widehat{{\tilde{v}}}_{j_2}||\widehat{{\tilde{v}}}_{j_3}||\widehat{{\tilde{v}}}_{k_2}|
 |\widehat{{\tilde{v}}}_{k_3}||\widehat{\tilde{h}}_k|\\
&\lesssim
\sum_{j_1,j_2,j_3,k_2,k_3}|k_m|^{-1/7+4\delta} |\widehat{{\tilde{v}}}_{1,j_1}||\widehat{{\tilde{v}}}_{2,j_2}||\widehat{{\tilde{v}}}_{3,j_3}||\widehat{{\tilde{v}}}_{4,k_2}|
 |\widehat{{\tilde{v}}}_{5,k_3}||\widehat{\tilde{h}}_k|\\
 &\lesssim \||k|^{1/2}v_1\|_{L^2}\prod_{i=2}^5\||k|^{-\delta}v_i\|_{L^1}\||k|^{-1/2}h\|_{L^2}\lesssim \prod_{i=1}^5\|v_i\|_{H^{1/2}}\|h\|_{H^{-1/2}}.
\end{align*}

This completes the proof of Lemma \ref{8-4}.
\end{proof}
We return to the analysis of the local well-posedness of \eqref{8-1}.
For the term \(\mathcal{R}(v)\) with \(s \geq 1/2\),
\begin{align*}
\| \mathbb{P}_{>N}e^{t\varepsilon\partial_x^2}\mathcal{R}(v) \|_{H^s}
&\lesssim \Big\|\sum_{k_1\neq 0} |k|^s \mathcal{F}_{k}[e^{t\varepsilon\partial_x^2}v] \cdot |k| \mathcal{F}_{-k}[e^{t\varepsilon\partial_x^2}v] \cdot \mathcal{F}_{k}[e^{t\varepsilon\partial_x^2}v] \Big\|_{L^2}\\
&\lesssim \|e^{t\varepsilon\partial_x^2}v\|_{H^s} \Big\| (|k|^{1/2}\mathcal{F}_k[e^{t\varepsilon\partial_x^2}v]) \ast (|k|^{1/2}\mathcal{F}_k[e^{t\varepsilon\partial_x^2}v]) \Big\|_{L^{\infty}} \\
&\lesssim \|e^{t\varepsilon\partial_x^2}v\|_{H^s} \|e^{t\varepsilon\partial_x^2}v\|^2_{H^{1/2}} \lesssim \|e^{t\varepsilon\partial_x^2}v\|^3_{H^s}.
\end{align*}

For the nonlinear terms \(\mathcal{N}_1(v)\) and \(\mathcal{N}_2(v)\) defined in \eqref{9.1}  and \eqref{9.2}, we have the bounds
\begin{align*}
|e^{t\varepsilon\partial_x^2}\mathcal{N}_1(v)|
&\lesssim \sum_{k_1+k_2+k_3=k} |k| |\tilde{\phi}|^{-1} |\mathcal{F}_{k_1}[e^{t\varepsilon \partial_x^2}v] \mathcal{F}_{k_2}[e^{t\varepsilon \partial_x^2}v] \mathcal{F}_{k_3}[e^{t\varepsilon \partial_x^2}v]|,\\
|e^{t\varepsilon\partial_x^2}\mathcal{N}_2(v)|
&\lesssim \sum_{\substack{k_1+k_2+k_3=k \\ j_1+j_2+j_3=k_1}} \frac{|k||k_1|}{|\tilde{\phi}|} |\mathcal{F}_{j_1}[e^{t\varepsilon \partial_x^2}v] \mathcal{F}_{j_2}[e^{t\varepsilon \partial_x^2}v] \mathcal{F}_{j_3}[e^{t\varepsilon \partial_x^2}v] \mathcal{F}_{k_2}[e^{t\varepsilon \partial_x^2}v] \mathcal{F}_{k_3}[e^{t\varepsilon \partial_x^2}v]|.
\end{align*}
Applying Lemmas \ref{0-6} and \ref{8-4}, we obtain the high-frequency estimates
\begin{align*}
\| \mathbb{P}_{>N}e^{t\varepsilon\partial_x^2}\mathcal{N}_1(v) \|_{H^s} &\lesssim N^{-1/2} \big\|e^{t\varepsilon\partial_x^2}v\big\|^3_{H^{s}},\\
\| \mathbb{P}_{>N}e^{t\varepsilon\partial_x^2}\mathcal{N}_2(v) \|_{H^s} &\lesssim \big\|e^{t\varepsilon\partial_x^2}v\big\|^5_{H^{s}}.
\end{align*}
Combining these estimates, we derive the inequality for the high-frequency component:
\begin{align*}
\|\mathbb{P}_{> N}e^{t\varepsilon\partial_x^2}v(t)\|_{H^s}
\lesssim \|v(0)\|_{H^s} + N^{-{1}/{2}} 
\big\|e^{t\varepsilon\partial_x^2}v
\big\|^3_{L^\infty_t H^s_x} + t \left( \big\|e^{t\varepsilon\partial_x^2}v\big\|^3_{L^\infty_t H^s_x} + \big\|e^{t\varepsilon\partial_x^2}v\big\|^5_{L^\infty_t H^s_x} \right).
\end{align*}

For the low-frequency component (\(|k| \leq N\)), by Lemma \ref{7-1} this yields
\begin{align*}
&\Big\| \mathbb{P}_{\leq N} e^{-t\varepsilon k^2} \sum_{k_1+k_2+k_3=k} ik e^{-i\tau Q_2(k,k_1,k_2,k_3)} v_{k_1} v_{k_2} v_{k_3} \Big\|_{H^s} \\
&\lesssim \Big\| \sum_{k_1+k_2+k_3=k} |k|^{s+1} \mathcal{F}_{k_1}[e^{t\varepsilon\partial_x^2}v] \mathcal{F}_{k_2}[e^{t\varepsilon\partial_x^2}v] \mathcal{F}_{k_3}[e^{t\varepsilon\partial_x^2}v] \Big\|_{L^2} \\
&\lesssim N \Big\| \sum_{k_1+k_2+k_3=k} |k|^{s} \mathcal{F}_{k_1}[e^{t\varepsilon\partial_x^2}v] \mathcal{F}_{k_2}[e^{t\varepsilon\partial_x^2}v] \mathcal{F}_{k_3}[e^{t\varepsilon\partial_x^2}v] \Big\|_{L^2} \\
&\lesssim N \ln N \big\|e^{t\varepsilon\partial_x^2}v\big\|^3_{H^s}\lesssim N^2 \big\|e^{t\varepsilon\partial_x^2}v\big\|^3_{H^s}, \quad \text{for} \>s\geq 1/2.
\end{align*}
Thus for \(|k| \leq N\) and \(s \geq 1/2\), we have
\begin{align*}
\|\mathbb{P}_{\leq N}e^{t\varepsilon\partial_x^2}v(t)\|_{H^s}
\lesssim \|v(0)\|_{H^s} + t N^2 \big\|e^{t\varepsilon\partial_x^2}v\big\|^3_{L^\infty_t H^s_x}.
\end{align*}
Combining the low and high-frequency estimates on $[0,T]$, we obtain
\begin{align*}
\|e^{t\varepsilon\partial_x^2} v \|_{L^\infty_t H_x^s}
&\lesssim \|v(0)\|_{H^s} + (N^{-{1}/{2}} + T N^2) \left( \big\|e^{t\varepsilon\partial_x^2}v\big\|^3_{L^\infty_t H^s_x} + \big\|e^{t\varepsilon\partial_x^2}v\big\|^5_{L^\infty_t H^s_x} \right).
\end{align*}
Using the relation \(v(t,x) = e^{t(\partial_x^3-\varepsilon\partial_x^2)}u(t,x)\), we translate this to the original solution
\begin{align*}
\|u\|_{L^\infty_t H^s}
&\lesssim \|u(0)\|_{H^s} + (N^{-{1}/{2}} + T N^2) \left( \big\|u\big\|^3_{L^\infty_t H^s_x} + \big\|u\big\|^5_{L^\infty_t H^s_x} \right).
\end{align*}
 Choosing \(N = T^{-{2}/{5}}\) to balance the terms, satisfying \(T N^2 = N^{-1/2}\) for \(t \in [0,T]\), then
\begin{align*}
\|u\|_{L^\infty_t H^s}
&\leq C \|u(0)\|_{H^s} + C T^{{1}/{5}} \|u\|_{L^\infty_t H^s_x} \left( \big\|u\big\|^2_{L^\infty_t H^s_x} + \big\|u\big\|^4_{L^\infty_t H^s_x} \right).
\end{align*}
To close the estimate, we select \(T > 0\) such that
\begin{align*}
C T^{{1}/{5}} \|u_0\|^2_{H^s} &\leq {1}/{8},\quad C T^{{1}/{5}} \|u_0\|^4_{H^s} \leq {1}/{8},
\end{align*}
which implies that
\begin{align}
C T^{{1}/{5}} \big\|u\big\|^2_{L^\infty_t H^s_x} \leq {1}/{4}, \quad
C T^{{1}/{5}} \big\|u\big\|^4_{L^\infty_t H^s_x} \leq {1}/{4}. \label{eq2}
\end{align}
This yields the uniform bound
\begin{align*}
\|u\|_{L^\infty_t H^s} \leq 2C \|u_0\|_{H^s} \lesssim \|\phi\|_{H^s},
\end{align*}
where the implicit constant is independent of \(\varepsilon\). This establishes local well-posedness with continuous dependence on initial data.

\section{Limit behavior for mKdV-B equation}
Finally, we prove the inviscid limit behavior in $H^s$. It is well known that mKdV equation \eqref{8-2} is completely integrable and has infinite conservation laws, and as a corollary one obtains that let $v$ be a smooth solution to \eqref{8-2} with initial data $v_0$, $k\in \mathbb{Z}_{+}$, we have
\begin{align*}
\sup_{t\in\mathbb{R}}\|v(t)\|_{H^k}\lesssim \|v_0\|_{H^{k}}.
\end{align*}
Let $u$ be a smooth solution of the mKdV-Burgers equation \eqref{8-1} with initial data $\phi$. However, there are less symmetries for \eqref{8-1}. We will expect that the $H^k$ norm of the solution remains dominated for a finite time $T>0$, since the dissipative term behaves well for $t>0$. Now we prove foe for $k=2$ which will suffice for our purpose.

We need the following variant of the Kato-Ponce commutator estimate:
\begin{align}\label{99-4}
\|[\partial_x, f]g\|_{L^2}\lesssim \|f_x\|_{L^{\infty}}\|g\|_{H^{s-1}}+\|f\|_{H^s}\|g\|_{L^{\infty}},\quad s>1/2.
\end{align}
Taking the $H^k$-scalar product with $u$ and using the Kato-Ponce commutator estimate we get
\begin{align*}
\frac{d}{dt}\|u(t)\|_{H^k}^2 &= 2\int_{\mathbb{T}} J_x^k u_t J_x^k u dx \\
&= 2\int_{\mathbb{T}} J_x^k\Big(-\partial_x^{3}u +\varepsilon \partial_x^{2} u + \partial_x
(u^3)\Big)J_x^k u dx\\
&= 2\int_{\mathbb{T}} \varepsilon J_x^k \partial_x^{2} u J_x^k u dx
 + 2\int_{\mathbb{T}} J_x^k\partial_x
(u^3)J_x^k u dx\\
&=-2\varepsilon\|\partial_xu\|^2_{H^k}+6\int_{\mathbb{T}} u^2\partial_x(J_x^ku)J_x^k u dx+6\int_{\mathbb{T}} [J_x^k, u^2]\partial_xuJ_x^k u dx\\
&\lesssim -2\varepsilon\|\partial_xu\|^2_{H^k}+\|\partial_x(u^2)\|_{L^{\infty}}\|u\|^2_{H^k}+\|u_x\|_{L^{\infty}}\|u^2\|_{H^k}\|u\|_{H^k}
\end{align*}
Thus we have
$$\frac{d}{dt}\|u(t)\|_{H^k}^2+\varepsilon\|u\|^2_{H^{k+1}}\lesssim \|u\|^4_{H^k},\quad k>3/2.$$
Taking $k=2$ and integrating the differential equation on $[0,T]$, we obtain
\begin{align}\label{99-3}
\sup_{t \in [0,T]} \| u(t) \|_{H^2} + \varepsilon^{1/2} \left( \int_0^T \| u(\tau) \|_{H^3}^2  d\tau \right)^{1/2} \leq C(T, \|\phi\|_{H^2}).
\end{align}

Assume $u$ is a solution to \eqref{8-1} obtained in the last section and $v$ is a solution to \eqref{8-2}, with initial data $\psi_1, \psi_2 \in H^2(\mathbb{T})$ respectively.
Let $w=u-v$, $w_0=u_0-v_0$, then $w$ satisfies
\begin{equation}
\begin{cases}
w_t + w_{xxx} -\varepsilon \partial_x^{2} u = \partial_x(w(v^2 + u^2 + v u)), & t \in \mathbb{R}_+, x \in \mathbb{T}, \\
w(0) = \phi.
\end{cases}
\end{equation}

We first view $\varepsilon \partial_x^{2} u$ as a perturbation to the difference equation of the mKdV equation.
Taking the $H^s$-scalar product with $w$ we get
\begin{align*}
\frac{d}{dt}\|w(t)\|_{H^s}^2 &= 2\int_{\mathbb{T}} J_x^s w_t J_x^s w dx \\
&= 2\int_{\mathbb{T}} J_x^s\Big(-\partial_x^{3}w +\varepsilon \partial_x^{2} u + \partial_x
(w(v^2 + u^2 + vu))\Big)J_x^s w dx\\
&= 2\int_{\mathbb{T}} J_x^s\varepsilon \partial_x^{2} u J_x^s w dx
 + 2\int_{\mathbb{T}} J_x^s\partial_x
(w(v^2 + u^2 + vu))J_x^s w dx:=I_1+I_2.
\end{align*}
The contribution of the first term can be obtained by integration by parts and Young's inequality,
\begin{align}\label{99-5}
I_1\lesssim \|\varepsilon \partial_x^{2+s} u\|_{L^2}\|J_x^s w\|_{L^2}\leq \frac{\varepsilon^2}{2}\|\partial_x^{2+s} u\|^2_{L^2}+\frac{1}{2}\|w\|^2_{H^s}.
\end{align}

Let $U=v^2 + u^2 + vu$.
Integrating by parts and applying the Kato-Ponce commutator estimate in \eqref{99-4}, we easily estimate the second term by
\begin{align*}
I_2&=
2\int_{\mathbb{T}} J_x^s(U\partial_xw)J_x^s w dx+2\int_{\mathbb{T}} J_x^s(w\partial_xU)J_x^s w dx\\
&=2\int_{\mathbb{T}} U\partial_x(J_x^sw)J_x^s w dx+2\int_{\mathbb{T}} [J_x^s, U]\partial_xwJ_x^s w dx
+2\int_{\mathbb{T}} J_x^s(w\partial_xU)J_x^s w dx\\
&= -\int_{\mathbb{T}} \partial_xU (J_x^sw)^2 dx+2\int_{\mathbb{T}} [J_x^s, U]\partial_xwJ_x^s w dx+2\int_{\mathbb{T}} J_x^s(w\partial_xU)J_x^s w dx\\
&\lesssim \|\partial_xU\|_{L^{\infty}}\|w\|^2_{H^s}+\|w\partial_xU\|_{H^s}\|w\|_{H^s}\\
&\lesssim \|w\|^2_{H^s}(\|\partial_xU\|_{L^{\infty}}+\|\partial_xU\|_{H^s}),
\end{align*}
where the last line by Lemma \ref{7-1}, $\|w\partial_xU\|_{H^s}\lesssim \|w\|_{H^s}\|\partial_xU\|_{H^s}.$
Then we obtain
\begin{align}\label{99-6}
I_2\lesssim (\|u\|^2_{H^{s+1}}+\|v\|^2_{H^{s+1}})\|w\|^2_{H^s}, \quad s>1/2.
\end{align}

Combining the inequalities of $I_1$ and $I_2$ in \eqref{99-5} and \eqref{99-6}, and taking $s=1$, we have
\begin{align*}
\frac{d}{dt}\|w(t)\|_{H^1}^2 \lesssim \varepsilon^2\| u\|^2_{L_t^{\infty}H_x^{3}}+\|w\|^2_{L_t^{\infty}H_x^1}
+(\|u\|^2_{L_t^{\infty}H_x^2}+\|v\|^2_{L_t^{\infty}H_x^2})\|w\|^2_{L_t^{\infty}H_x^1},
\end{align*}
 Integrating the differential equation on $[0,T]$,  we obtain the following inequality
 \begin{align*}
\sup_{t\in[0,T]}\|w(t)\|^2_{H^1}&\lesssim \|w_0\|^2_{H^1}+ \varepsilon^2\int_0^T\|u\|^2_{H^3}d\tau+(\|u\|^2_{L_t^{\infty}H_x^2}+\|v\|^2_{L_t^{\infty}H_x^2}+1)
\int_0^T\|w\|^2_{H_x^1}d\tau.
\end{align*}
Since from \eqref{eq2} and Theorem 1.3,
$$T\|u\|^2_{L_t^{\infty}H_x^2}\leq 1/4, \quad T\|v\|^2_{L_t^{\infty}H_x^2}\leq 1/4,$$
and from \eqref{99-3},
\[
\varepsilon^{1/2} \left(\int_0^T \|u\|_{H^3}^2 d\tau\right)^{1/2} \leq C(T, \|\phi\|_{H^2}).
\]
Thus, for  $\psi_1,\psi_2\in H^2$, we immediately get that there exits $T>0$ suitably small such that
\begin{align*}
\|w(t)\big\|_{L_t^{\infty}H_x^1}\lesssim \|\psi_1-\psi_2\|_{H^1}+\varepsilon^{1/2} C(T,\|\psi_1\|_{H^2},\|\psi_2\|_{H^2} ).
\end{align*}

For fixed $T>0$, we need to prove that $\forall\eta>0$, there exists $0<\varepsilon<1,$ then
$$\|{S}_m^{\varepsilon}(\varphi) - {S}_m(\varphi)\|_{C([0,T];H^s)}<\eta.$$
Fixing this $K$ large enough, for $s\geq 1/2$, we get that
\begin{align}\label{8-8}
\|{S}_m^{\varepsilon}(\mathbb{P}_{\leq K}\varphi) - {S}_m(\mathbb{P}_{\leq K}\varphi)\|_{C([0,T];H^s)}\lesssim \varepsilon^{1/2} C(T,K,\|\varphi\|_{H^s} ).
\end{align}
Similar to the uniform continuity of \eqref{00} argument, we easily obtain that for $u_1, u_2\in H^s$, then
\begin{align}
\|u_1 - u_2\|_{H^s} \leq C(T, \|\phi_1\|_{H^s}, \|\phi_2\|_{H^s}) \|\phi_1 - \phi_2\|_{H^s},
\end{align}
Then it follows from the above uniform continuity
\begin{align*}
&\|S_m^{\varepsilon}(\varphi) - S_m^{\varepsilon}(\mathbb{P}_{\leq K}\varphi)\|_{C([0,T];L^2)}\leq \eta/4,\\
&\|S_m(\varphi) - S_m(\mathbb{P}_{\leq K}\varphi)\|_{C([0,T];L^2)}\leq \eta/4.
\end{align*}
Thus from Theorem \ref{8-6} and \eqref{8-8}, we get
\begin{align*}
\|{S}_m^{\varepsilon}(\varphi) - {S}_m(\varphi)\|_{C([0,T];H^s)}
&\leq \|{S}_m^{\varepsilon}(\varphi) - {S}_m^{\varepsilon}(\mathbb{P}_{\leq K}\varphi)\|_{C([0,T];H^s)} \\
&\quad + \|{S}_m^{\varepsilon}(\mathbb{P}_{\leq K}\varphi) - {S}_m(\mathbb{P}_{\leq K}\varphi)\|_{C([0,T];H^s)}\\
&\quad + \|{S}_m(\mathbb{P}_{\leq K}\varphi) - {S}_m(\varphi)\|_{C([0,T];H^s)}\leq \eta.
\end{align*}
The proof of Theorem \ref{8-5} is completed.
\section*{Appendix}

\setcounter{thm}{0}
\setcounter{equation}{0}
\renewcommand{\theremark}{A.\arabic{remark}}
\renewcommand{\theequation}{A.\arabic{equation}}
\renewcommand{\thethm}{A.\arabic{thm}}

Assume $u_{\varepsilon}$ is a solution to \eqref{00} and $u$ is a solution to \eqref{1.4}. This appendix provides the detailed proof for the estimate of the term $\int_0^t\int_{\mathbb{T}}\partial_x u \cdot w^2  dx  ds$ referenced in Section 4 of the main text. We establish the following inequality:
$$\int_0^t\int_{\mathbb{T}}\partial_xu\cdot w^2dxds\lesssim \varepsilon^2\int_0^T \|\partial_xu_{\varepsilon}(\tau)\|^2_{L^2}d\tau+C(T, \|u\|_{L^2},\|u_{\varepsilon}\|_{L^2})\|w\|^2_{L^2}.$$

 As H\"{o}lder and Young's inequalities are insufficient for our purpose, we develop an alternative approach through variable transformations.

Define $h = e^{t \partial_x^3} u$, $v = e^{t (\partial_x^3 - \varepsilon \partial_x^2)} u_\varepsilon$, $z = e^{t \partial_x^3} w$,
we have the following expressions:
\begin{align}
\partial_t h &=  e^{t \partial_x^3} \partial_x \big(e^{-t \partial_x^3} h\big)^2,\quad
\partial_t v = e^{t (\partial_x^3 - \varepsilon \partial_x^2)} \partial_x \big(e^{-t (\partial_x^3 - \varepsilon \partial_x^2)} v\big)^2,\label{1.1}\\
\partial_t z
&=  \varepsilon \partial_x^2\big(e^{t \varepsilon\partial_x^2}v\big)
+ e^{t\partial_x^3 } \partial_x \Big(\big(e^{-t (\partial_x^3 - \varepsilon \partial_x^2)} v
+e^{-t\partial_x^3} h\big)\cdot e^{-t\partial_x^3} z\Big).\label{1.2}
\end{align}
We note that
\begin{align*}
\int_{\mathbb{T}} \partial_x u w^2 \, dx&= \sum_{k_1 + k_2 + k_3 = 0 } ik_1\hat{u}_{k_1} \hat{w}_{k_2}  \hat{w}_{k_3} = \sum_{k_1 + k_2 + k_3 = 0 } ik_1 e^{it(k_1^3+k_2^3+k_3^3)} \hat{h}_{k_1} \hat{z}_{k_2} \hat{z}_{k_3} \\
&= \sum_{k_1 + k_2 + k_3 = 0 } ik_1 e^{3itk_1k_2k_3} \hat{h}_{k_1} \hat{z}_{k_2} \hat{z}_{k_3}.
\end{align*}
Integration by parts since $\partial_t(\frac{e^{3it k_1k_2k_3}}{3ik_1k_2k_3})$,  we have
\begin{align*}
&\int_0^t\int_{\mathbb{T}} \partial_x u w^2 \, dx ds\\
&=\sum_{k_1 + k_2 + k_3 = 0} \frac{e^{3isk_1k_2k_3}}{3k_2k_3} \hat{h}_{k_1} \hat{z}_{k_2} \hat{z}_{k_3}\Big|^{s=t}_{s=0}\\
&\quad-\int_0^t\sum_{k_1 + k_2 + k_3=0} \frac{e^{3is k_1k_2k_3}}{3k_2k_3} \Big(\partial_s\hat{h}_{k_1}\hat{z}_{k_2} \hat{z}_{k_3}
+\partial_s\hat{z}_{k_2}\hat{h}_{k_1} \hat{z}_{k_3}+\partial_s\hat{z}_{k_3}\hat{h}_{k_1}\hat{z}_{k_2} \Big)ds\\
&=: L_1+L_2+L_3+L_4.
\end{align*}
By H\"{o}lder inequality, we have
\begin{align}\label{00-8}
|L_1|\lesssim \|h\|_{L^2}\|\partial_x^{-1}z\|_{L^2}\|\partial_x^{-1}z\|_{L^{\infty}}
\lesssim \|h\|_{L^2}\|z\|^2_{L^2}=\|u\|_{L^2}\|w\|^2_{L^2}.
\end{align}
For $L_2$, using the definition of $\partial_s\hat{h}$ from equation \eqref{1.1}, and $k_1=-k_2-k_3$, we have
\begin{align*}
L_2
&=-\int_0^t\sum_{k_1 + k_2 + k_3=0} \frac{e^{3is k_1k_2k_3}}{3k_2k_3} \sum_{k_1=k_{\alpha}+k_{\beta}}ik_1 e^{-is(k_1^3-k_{\alpha}^3-k_{\beta}^3)}\hat{h}_{k_{\alpha}}\hat{h}_{k_{\beta}}\hat{z}_{k_2} \hat{z}_{k_3}ds\\
&=\int_0^t\sum_{k_{\alpha}+k_{\beta}+ k_2 + k_3=0} \frac{i(k_2+k_3)e^{-3is \phi
_1}}{3k_2k_3}\hat{h}_{k_{\alpha}}\hat{h}_{k_{\beta}}\hat{z}_{k_2} \hat{z}_{k_3}ds\\
&=\frac{i}{3}\int_0^t\sum_{k_{\alpha}+k_{\beta} + k_2 + k_3 = 0}
e^{-3is \phi_1}\left(\frac{1}{k_2}+\frac{1}{k_3}\right)\hat{h}_{k_{\alpha}}\hat{h}_{k_{\beta}}\hat{z}_{k_2} \hat{z}_{k_3}ds,
\end{align*}
where $\phi_1=-(k_{\alpha}+k_{\beta})k_2k_3+k_{\alpha}k_{\beta}k_1=(k_{\alpha}+k_{\beta})(k_{\alpha}+k_{3})(k_{\beta}+k_3).$

 Renaming the variables $k_1=k_{\alpha}$, $k_2=k_{\beta}$, $k_3=k_2$, $k_4=k_3$.
Since $k_1, k_2, k_3, k_4$ are symmetric, we have the following relation:
\begin{align}\label{0.02}
\frac{1}{k_1}+ \frac{1}{k_2}+ \frac{1}{k_3}+ \frac{1}{k_4}=\frac{-\phi_1}{k_1k_2k_3k_4}.
\end{align}

If $\phi_1=0$, we work with mean zero initial data for KdV, which implies $\mathbb{P}_0v=0$. Consequently, we have $k_{\alpha}+k_{\beta}\neq 0$, leading to $(k_{\alpha}+k_{3})(k_{\beta}+k_3)=0$. Therefore
\begin{align*}
|L_2|&\lesssim \int_0^t\sum_{k_{\alpha}+k_{\beta} + k_2 + k_3 = 0}\hat{h}_{k_{\alpha}}\hat{h}_{k_{\beta}}\hat{z}_{k_2} \hat{z}_{k_3}ds\\
&\lesssim\int_0^t\mathcal{F}_0[h z]\mathcal{F}_0[h z]ds\lesssim t\|h\|^2_{L^2}\|z\|^2_{L^2}\lesssim t\|u\|^2_{L^2}\|w\|^2_{L^2}.
\end{align*}
If $\phi_1\neq0$, we can use integration by parts since $\partial_t(\frac{ie^{-3is \phi_1}}{3\phi_1})$, then $L_2$ can be rewritten as
\begin{align*}
L_2&=\int_0^t\sum_{k_1 + k_2 + k_3 + k_4 = 0} \frac{e^{-3is \phi_1}}{6i}
\frac{\phi_1}{k_1k_2k_3k_4}\hat{h}_{k_1}\hat{h}_{k_2}\hat{z}_{k_3} \hat{z}_{k_4}ds\nonumber\\
&=\frac{1}{18}\sum_{k_1 + k_2 + k_3 + k_4 = 0} \frac{e^{-3is \phi_1}}{k_1k_2k_3k_4}\hat{h}_{k_1}\hat{h}_{k_2}\hat{z}_{k_3} \hat{z}_{k_4}\Big|_{0}^t\nonumber\\
&\quad-\frac{1}{18}\int_0^t\sum_{k_1 + k_2 + k_3 + k_4 = 0}\frac{e^{-3is \phi_1}}{k_1k_2k_3k_4}\Big(
\partial_{s}\hat{h}_{k_1}\hat{h}_{k_2}\hat{z}_{k_3} \hat{z}_{k_4}+\partial_{s}\hat{h}_{k_2}\hat{h}_{k_1}\hat{z}_{k_3} \hat{z}_{k_4}\\&\qquad\qquad\qquad\qquad\qquad\qquad\quad+
\partial_{s}\hat{z}_{k_3}\hat{h}_{k_1}\hat{h}_{k_2} \hat{z}_{k_4}+
\partial_{s}\hat{z}_{k_4}\hat{h}_{k_1}\hat{h}_{k_2} \hat{z}_{k_3} \Big)ds\\
&=:L_2^1+L_2^2+L_2^3+L_2^4+L_2^5.
\end{align*}
By H\"{o}lder inequality, we have
\begin{align*}
|L_2^1|\lesssim \|h\|^2_{L^2}\|z\|^2_{L^2}=\|u\|^2_{L^2}\|w\|^2_{L^2}.
\end{align*}
From the definition of $\partial_t\hat{h}_k$ in \eqref{1.1}, we can derive
\begin{align*}\|\partial_x^{-1}\partial_th\|_{L^1}\lesssim \|e^{-t \partial_x^3}h\cdot e^{-t \partial_x^3}h\|_{L^1}\lesssim \|h\|^2_{L^2}=\|u\|^2_{L^2}.
\end{align*}
Therefore,
\begin{align*}
&|L_{2}^2+L_{2}^3|
\lesssim |L_{2}^2|+|L_{2}^3|\lesssim t\|\partial_x^{-1}\partial_th\|_{L^1}\|\partial_x^{-1}h\|_{L^{\infty}}
\|\partial_x^{-1}z\|^2_{L^{\infty}}\lesssim t\|u\|^3_{L^2}\|w\|^2_{L^2}.
\end{align*}
Substituting the expression for $\partial_t\hat{z}_k$ from \eqref{1.2} into $L_{2}^4$,
\begin{align*}
L_{2}^4&=\frac{1}{18}\int_0^t\sum_{k_1 + k_2 + k_3 + k_4 = 0}\frac{e^{-3is \phi_1}}{k_1k_2k_4} \varepsilon k_3e^{-t\varepsilon k_3^2}\hat{v}_{k_3}\hat{h}_{k_1}\hat{h}_{k_2}\hat{z}_{k_4}ds\\
&\quad-\frac{1}{36}\int_0^t\sum_{k_1 + k_2 + k_{\alpha}+k_{\beta}+k_4 =0}\frac{ ie^{-3is \phi_2}}{k_1k_2k_4} \Big(e^{-t\varepsilon k_{\alpha}^2}\hat{v}_{k_{\alpha}}+ \hat{h}_{k_{\alpha}}\Big)\hat{z}_{k_{\beta}}\hat{h}_{k_1}\hat{h}_{k_2} \hat{z}_{k_4}ds,
\end{align*}
where $\phi_2=\phi_1+k_3k_{\alpha}k_{\beta}$.
Then
\begin{align*}
|L_{2}^4|+|L_{2}^5|
&\lesssim t\varepsilon\|\partial_xu_{\varepsilon}\|_{L^2}\|u\|^2_{L^2}\|w\|_{L^2}+t(\|u_{\varepsilon}\|_{L^2}+\|u\|_{L^2})\|u\|^2_{L^2}\|w\|^2_{L^2}.
\end{align*}
By combining the five inequalities above, we arrive at the following result:
\begin{align}\label{00-9}
|L_2|\lesssim \|u\|^2_{L^2}\|w\|^2_{L^2}+ t\varepsilon\|\partial_xu_{\varepsilon}\|_{L^2}\|u\|^2_{L^2}\|w\|_{L^2}+t(\|u_{\varepsilon}\|_{L^2}+\|u\|_{L^2})\|u\|^2_{L^2}\|w\|^2_{L^2}.
\end{align}

Since $L_3$ is similar to $L_4$, we only estimate one of them. Without loss of generality, we focus on $L_3$ and proceed with the procedure used in the proof of $L_2$.

We substitute the expression for $\partial_s\hat{z}_k$ from \eqref{1.2} into $L_3$, which yields that
\begin{align*}
L_3
&=\int_0^t\sum_{k_1 + k_2 + k_3=0} \frac{e^{3is k_1k_2k_3}}{3k_3} \varepsilon k_2e^{-s\varepsilon k_2^2}\hat{v}_{k_2}\hat{h}_{k_1}\hat{z}_{k_3}ds\nonumber\\
&\quad+\frac{1}{3i}\int_0^t\sum_{k_1 + k_2 + k_3 + k_4 = 0} \frac{e^{-3is \phi_1}}{k_4}\hat{h}_{k_1}\big(e^{-s\varepsilon k_2^2}\hat{v}_{k_2}+ \hat{h}_{k_2}\big)\hat{z}_{k_3}\hat{z}_{k_4}ds.
\end{align*}
Then the second term on the right-hand side of $L_3$ can be bounded by the following expression:
\begin{align*}
&\int_0^t\sum_{k_1 + k_2 + k_3 + k_4 = 0} \frac{e^{-3is \phi_1}}{12i}
\Big(\frac{1}{k_1}+\frac{1}{k_2}+\frac{1}{k_3}+\frac{1}{k_4}\Big)\hat{h}_{k_1}\big(\hat{v}_{k_2}+ \hat{h}_{k_2}\big)\hat{z}_{k_3}\hat{z}_{k_4}ds\nonumber\\
&=-\frac{1}{36}\sum_{k_1 + k_2 + k_3 + k_4 = 0}
\frac{e^{-3is \phi_1}}{k_1k_2k_3k_4}\hat{h}_{k_1}\big(\hat{v}_{k_2}+ \hat{h}_{k_2}\big)\hat{z}_{k_3}\hat{z}_{k_4}\Big|_{0}^t\\
&\quad+\frac{1}{36}\int_0^t\sum_{k_1 + k_2 + k_3 + k_4 = 0}
\frac{e^{-3is \phi_1}}{k_1k_2k_3k_4}\partial_s\Big(\hat{h}_{k_1}\big(\hat{v}_{k_2}+ \hat{h}_{k_2}\big)\hat{z}_{k_3}\hat{z}_{k_4}\Big)ds\\
&=:L_{3}^1+L_{3}^2+L_{3}^3+L_{3}^4+L_{3}^5.
\end{align*}
Following the proof of $L_3$ and by the H\"{o}lder inequality, we have
\begin{align*}
&|L_{3}^1|+|L_{3}^2|+|L_{3}^3|+|L_{3}^4|+|L_{3}^5|\\
&\lesssim  (\|u_{\varepsilon}\|_{L^2}+\|u\|_{L^2})\|u\|_{L^2}\|w\|^2_{L^2}
+\int_0^t\varepsilon \|\partial_xu_{\varepsilon}\|_{L^2}(\|u_{\varepsilon}\|_{L^2}+\|u\|_{L^2})\|u\|_{L^2}\|w\|_{L^2}d\tau\\
&\quad+t(\|u_{\varepsilon}\|^2_{L^2}\|u\|_{L^2}+\|u_{\varepsilon}\|_{L^2}\|u\|^2_{L^2}+\|u\|^3_{L^2})\|w\|^2_{L^2}.
\end{align*}
Then we obtain
\begin{align}
|L_3|&\lesssim t\varepsilon \|\partial_xu_{\varepsilon}\|_{L_t^{\infty}L_x^2}\|u\|_{L_t^{\infty}L_x^2}\|w\|_{L_t^{\infty}L_x^2}
+|L_{3}^1|+|L_{3}^2|+|L_{3}^3|+|L_{3}^4|+|L_{3}^5|\nonumber\\
&\lesssim \int_0^t\varepsilon^2\|\partial_xu_{\varepsilon}\|^2_{L^2}d\tau
+t(\|u_{\varepsilon}\|^2_{L_t^{\infty}L_x^2}\|u\|_{L_t^{\infty}L_x^2}+\|u_{\varepsilon}\|_{L_t^{\infty}L_x^2}\|u\|^2_{L_t^{\infty}L_x^2}
+\|u\|^3_{L_t^{\infty}L_x^2})\|w\|^2_{L_t^{\infty}L_x^2}\nonumber\\
&\quad+(\|u_{\varepsilon}\|_{L_t^{\infty}L_x^2}\|u\|_{L_t^{\infty}L_x^2}+\|u\|_{L_t^{\infty}L_x^2}^2)\|w\|^2_{L_t^{\infty}L_x^2}.\label{00-10}
\end{align}
Combining the results for $L_1$, $L_2$, $L_3$ and $L_4$ from \eqref{00-8}, \eqref{00-9} and \eqref{00-10}, we derive
\begin{align*}
&\int_0^t\int_{\mathbb{T}} \partial_x u w^2 \, dx ds
\lesssim \varepsilon^2\int_0^T \|\partial_xu_{\varepsilon}(\tau)\|^2_{L^2}d\tau+C(T, \|u\|_{L^2},\|u_{\varepsilon}\|_{L^2})\|w\|^2_{L^2}.
\end{align*}

\section*{Acknowledgments}
The authors would like to thank Professor Yifei Wu for helpful discussions and comments. The authors are partially supported by NSFC 12171356.


\begin{thebibliography}{00}

	\bibitem{J}
	{\sc
    J. Bourgain},
	\textit{Fourier transform restriction phenomena for certain lattice subsets and applications to nonlinear evolution equations. II. The KdV equation},
	Geom. Funct. Anal. \textbf{3} (1993) 209--262.
	
	\bibitem{Li-KatoPonce}
    {\sc
	J. Bourgain, D. Li},
	\textit{On an endpoint Kato-Ponce inequality},
	Differential Integral Equations \textbf{27} (2014) 1037--1072.

	\bibitem{X-M}
	{\sc
		 X. Carvajal, M. Panthee},
	\textit{Sharp local well-posedness of KdV type equations with dissipative perturbations},
	Q. Appl. Math. \textbf{74}(3) (2016) 571--594.
	
	\bibitem{JM}
	{\sc J. Colliander, M. Keel, G. Staffilani, H. Takaoka, T. Tao},
	\textit{Sharp global well-posedness for KdV and modified KdV on $\mathbb{R}$ and $\mathbb{T}$},
	J. Amer. Math. Soc. \textbf{16} (2003) 705--749.
	
	\bibitem{ZST}
	{\sc Z. Guo, S. Kwon, T. Oh},
	\textit{Poincar\'e-Dulac normal form reduction for unconditional well-posedness of the periodic cubic NLS},
	Comm. Math. Phys. \textbf{322} (2013) 19--48.
	
	\bibitem{ZB}
	{\sc Z. Guo, B. Wang},
	\textit{Global well-posedness and inviscid limit for the Korteweg-de Vries-Burgers equation},
	J. Differential Equations \textbf{246} (2009) 3864--3901.
	
	\bibitem{ZH}
	{\sc H. Zhang, L. Han},
	\textit{Global well-posedness and inviscid limit for the modified Korteweg-de Vries-Burgers equation},
	Nonlinear Anal. \textbf{71} (2009) 1708--1715.
	
	\bibitem{ST}
	{\sc S. Kwon, T. Oh},
	\textit{On unconditional well-posedness of modified KdV},
	Int. Math. Res. Not. IMRN \textbf{15} (2012) 3509--3534.
	
	\bibitem{T}
	{\sc T. Kato},
	\textit{On nonlinear Schr\"odinger equations. II. $H^s$-solutions and unconditional well-posedness},
	J. Anal. Math. \textbf{67} (1995) 281--306.
	
	\bibitem{TG}
	{\sc T. Kato, G. Ponce},
	\textit{Commutator estimates and the Euler and Navier-Stokes equations},
	Commun. Pure Appl. Math. \textbf{41} (1988) 891--907.
	
	\bibitem{K}
	{\sc T. Kappeler, P. Topalov},
	\textit{Global well-posedness of KdV in $H^{-1}(\mathbb{T}, \mathbb{R})$},
	Duke Math. J. \textbf{135} (2006) 327--360.
	
	\bibitem{KT}
	{\sc T. Kappeler, P. Topalov},
	\textit{Global well-posedness of mKdV in $L^2(\mathbb{T}, \mathbb{R})$},
	Comm. Partial Differential Equations \textbf{30} (2005) 435--449.
	
	\bibitem{CGL}
	{\sc C. Kenig, G. Ponce, L. Vega},
	\textit{A bilinear estimate with applications to the KdV equation},
	J. Amer. Math. Soc. \textbf{9} (1996) 573--603.
	
	\bibitem{RM}
	{\sc R. Killip, M. Visan},
	\textit{KdV is well-posed in $H^{-1}$},
	Ann. Math. (2) \textbf{190}(1) (2019) 249--305.
	
	\bibitem{BY}
	{\sc B. Li, Y. Wu},
	\textit{An unfiltered low-regularity integrator for the KdV equation with solutions below $H^1$} [Preprint],
	arXiv:2206.09320 (2022).
	
	\bibitem{Li-D}
	{\sc D. Li},
	\textit{On Kato-Ponce and fractional Leibniz},
	Rev. Mat. Iberoam. \textbf{35} (2019) 23--100.
	
	\bibitem{LM}
	{\sc L. Molinet, F. Ribaud},
	\textit{The Cauchy problem for dissipative Korteweg de Vries equations in Sobolev spaces of negative order},
	Indiana Univ. Math. J. \textbf {50}(4) (2001) 1745--1776.
	
	\bibitem{L-F}
	{\sc L. Molinet, F. Ribaud},
	\textit{On the low regularity of the Korteweg-de Vries-Burgers equation},
	Int. Math. Res. Not. \textbf{2002}(37) (2002) 1979--2005.
	
	\bibitem{N}
	{\sc K. Nakanishi, H. Takaoka, Y. Tsutsumi},
	\textit{Local well-posedness in low regularity of the mKdV equation with periodic boundary condition},
	Discrete Contin. Dyn. Syst. \textbf{28}(4) (2010) 1635--1654.
	
	\bibitem{TT}
	{\sc H. Takaoka, Y. Tsutsumi},
	\textit{Well-posedness of the Cauchy problem for the modified KdV equation with periodic boundary condition},
	Int. Math. Res. Not. \textbf{2004} (2004) 3009--3040.
	
	\bibitem{V}
	{\sc A. V. Babin, A. Ilyin, E. S. Titi},
	\textit{On the regularization mechanism for the periodic Korteweg-de Vries equation},
	Commun. Pure Appl. Math. \textbf{64} (2011) 591--648.
	
	\bibitem{Wang}
	{\sc B. Wang},
	\textit{The limit behavior of solutions for the Cauchy problem of the complex Ginzburg-Landau equation},
	Commun. Pure Appl. Math. \textbf{53} (2002) 481--508.
	
	\bibitem{WZ}
	{\sc Y. Wu, X. Zhao},
	\textit{Embedded exponential-type low-regularity integrators for KdV equation under rough data} [Software],
	BIT \textbf{61} (2021). https://doi.org/10.1007/s10543-021-00895-8.
	
	\bibitem{Zhou}
	{\sc Y. Zhou},
	\textit{Uniqueness of weak solution of the KdV equation},
	Int. Math. Res. Not. \textbf{6} (1997) 271--283.
\end{thebibliography}
\end{document}